\documentclass[12pt,letterpaper]{amsart}

\newtheorem{theorem}{Theorem}[section]
\newtheorem{lemma}[theorem]{Lemma}
\newtheorem{proposition}[theorem]{Proposition}
\newtheorem{corollary}[theorem]{Corollary}

\theoremstyle{definition}

\theoremstyle{remark}

\numberwithin{equation}{section}

\usepackage{latexsym}
\usepackage{amssymb}
\usepackage{amsmath}
\usepackage{amscd}
\usepackage{mathrsfs}

  \DeclareMathOperator{\Span}{{\rm span}}
  \DeclareMathOperator{\GL}{{\rm GL}} \DeclareMathOperator{\SL}{{\rm
      SL}} 
  \DeclareMathOperator{\End}{{\rm End}}

  \DeclareMathOperator{\norL}{{\rm N}_L}
  \DeclareMathOperator{\nor}{{\rm N}_G} 
  
  \DeclareMathOperator{\noroneL}{{\rm N}^1_L}
  \DeclareMathOperator{\supp}{{\rm supp}} 
  \DeclareMathOperator{\Stab}{{\rm Stab}} 
  \DeclareMathOperator{\Ad}{{\rm Ad}} 
  \DeclareMathOperator{\ad}{{\rm ad}} 
  \DeclareMathOperator{\diag}{\rm{diag}} 
  \DeclareMathOperator{\Lie}{{\rm Lie}} 
  \DeclareMathOperator{\Cc}{{\rm C}_{{\rm c}}} 

   \renewcommand{\phi}{\varphi}

 \newcommand{\vect}[1]{{\boldsymbol{#1}}}

   \newcommand{\vt}{\vect{t}}
   \newcommand{\vv}{\vect{v}}
  \newcommand{\vw}{\vect{w}}
  \newcommand{\vx}{\vect{x}}

  \newcommand{\vmu}{\vect{\mu}}
  \newcommand{\vnu}{\vect{\nu}}
  \newcommand{\vxi}{\vect{\xi}}
  
  \newcommand{\vtau}{\vect{\tau}}

  \newcommand{\vzero}{\vect{0}}

  \newcommand{\field}[1]{\mathbb{#1}} 
  \newcommand{\R}{\field{R}} \newcommand{\N}{\field{N}}
   \newcommand{\Z}{\field{Z}}

  \providecommand{\abs}[1]{\lvert#1\rvert}
  \providecommand{\Abs}[1]{\left\lvert #1 \right\rvert}
  \providecommand{\norm}[1]{\lVert#1\rVert}

  \providecommand{\trn}[1]{{\,^{\bf t}\!#1}}

  \renewcommand{\setminus}{\smallsetminus}

  \newcommand{\inv}{^{-1}}

  \newcommand{\tcup}[1]{\textstyle{\bigcup_{#1}}\,}

  \newcommand{\psmat}[1]{\bigl(\begin{smallmatrix} #1
    \end{smallmatrix}\bigr)}

  \newcommand{\toinfty}{\stackrel{i\to\infty}{\longrightarrow}}

  \newcommand{\la}[1]{\mathfrak{\lowercase{#1}}}

  \newcommand{\cA}{\mathcal{A}} \newcommand{\sA}{\mathscr{A}}

  \newcommand{\cB}{\mathcal{B}} 

   \newcommand{\cC}{\mathcal{C}}

  \newcommand{\cD}{\mathcal{D}}

  \newcommand{\cE}{\mathcal{E}}

  \newcommand{\cF}{\mathcal{F}} 

   \newcommand{\sH}{\mathscr{H}}

  \newcommand{\cK}{\mathcal{K}}

  \newcommand{\cN}{\mathcal{N}} \newcommand{\sN}{\mathscr{N}}

  \newcommand{\cO}{\mathcal{O}}

  \newcommand{\cT}{\mathcal{T}}  \newcommand{\sT}{\mathscr{T}}

  \newcommand{\cW}{\mathcal{W}} \newcommand{\sW}{\mathscr{W}}
  \newcommand{\frakW}{\mathfrak{W}} 
\newcommand{\sw}{\mathfrak{w}}

\newcommand{\lml}{{L/\Lambda}}

\begin{document}

  \title[Expanding translates of curves]{Expanding translates of curves and
    Dirichlet-Minkowski theorem on linear forms}
  \author{Nimish A. Shah}
  \address{Tata Institute of Fundamental Research, Mumbai 400005,
    INDIA}
  \email{nimish@math.tifr.res.in}

  \date{February 22, 2008}



  \begin{abstract}
    We show that a multiplicative form of Dirichlet's theorem on
    simultaneous Diophantine approximation as formulated by Minkowski,
    cannot be improved for almost all points on any analytic curve on
    $\R^k$ which is not contained in a proper affine subspace. Such an
    investigation was initiated by Davenport and Schmidt in the late
    sixties.
    
    The Diophantine problem is then settled via showing that certain
    sequence of expanding translates of curves on the homogeneous
    space of unimodular lattices in $\R^{k+1}$ gets equidistributed in
    the limit. We use Ratner's theorem on unipotent flows,
    linearization techniques, and a new observation about intertwined
    linear dynamics of various $\SL(m,\R)$'s contained in
    $\SL(k+1,\R)$.
  \end{abstract}

  \maketitle

  \section{Introduction}
  Extending Dirichlet's theorem (1842) on simultaneous Diophantine
  approximation in various forms, Minkowski (1896) proved the
  following theorem as a consequence of his convex body theorem
  \cite[Chap.-II]{Schmidt:Springer-lect}:

  \subsection*{Minkowski's theorem on linear forms}. \emph{Let
    $(\phi_{ij})\in\SL(n,\R)$ and $\alpha_1,\dots,\alpha_n$ be
    positive numbers with $\alpha_1\cdots \alpha_n=1$. Then there
    exist integers $x_1,\dots,x_n$, not all zero,  such that
    \begin{align}
      \label{eq:88}
      \begin{split}
      \abs{\phi_{11} x_1 +\dots+\phi_{1n} x_n}&\leq \alpha_1\\
      \abs{\phi_{i1} x_1+\dots+\phi_{in} x_{in}}&< \alpha_i\quad
      (2\leq i\leq n).
    \end{split}
  \end{align}}

By putting $\phi_{11}=\dots=\phi_{nn}=1$, and $\phi_{ij}=0$ for $i\neq
j$ and $i\geq 2$, we obtain a multiplicative variation of Dirichlet's
theorem: {\em Given $(\xi_1,\dots,\xi_k)\in\R^k$ and positive integers
  $N_1,\dots, N_k$, there exist integers $q_1,\dots,q_k$ and $p$, not
  all zero, such that
  \begin{align}
    \label{eq:93}
    \abs{q_1\xi_1+\dots+q_k\xi_k-p}\leq (N_1\cdots N_k)\inv, \quad
    \abs{q_i}<N_i \quad (1\leq i\leq k).
  \end{align}}

Following Davenport and Schmidt~\cite{Davenport+Schmidt:Dirichlet}, we
say that given any infinite set $\cN\subset (\Z_+)^k$, the Dirichlet's
theorem (DT) cannot be improved along $\cN$ for
$(\xi_1,\dots,\xi_k)\in\R^k$, if for every $0<\mu<1$ there are
infinitely many $(N_1,\dots,N_k)\in\cN$ such that the following system
of inequalities is insoluble for integers $q_1,\dots,q_k$ and $p$, not
all zero:
  \begin{equation}
    \label{eq:100}
    \abs{q_1\xi_1+\dots+q_k\xi_k-p}\leq \mu (N_1\cdots N_k)\inv, \quad
    \abs{q_i}<\mu N_i\quad (1\leq j\leq k).
  \end{equation}

  Davenport and Schmidt~\cite{Davenport+Schmidt:Dirichlet} showed that
  for $\cN=\{(N,\dots,N)\in\Z^k:N\in\Z_+\}$, the DT cannot be improved
  along $\cN$ for almost all points of $\R^k$. The same conclusion was
  obtained by Kleinbock and Weiss~\cite{Kleinbock+Weiss:Dirichlet} for
  sets $\cN\subset\Z^k$ with infinite projection on each coordinate.

  In fact, Davenport and Schmidt~\cite{DS:curve} showed that for
  $k=2$, and for almost every $\xi\in\R$, the inequalities
  \eqref{eq:100} for $(\xi_1,\xi_2)=(\xi,\xi^2)$ do not have nonzero
  integral solution for infinitely many $N_1=N_2$. Such results for
  related quantities, say for points on a certain type of curve or a
  submanifold, were subsequently generalized in
  \cite{Baker:curves,Dodson:manifolds,Bugeaud:poly,Kleinbock+Weiss:Dirichlet};
  in each case for $\mu\leq \mu_0$ for some small explicit value of
  $\mu_0<1$ depending on the curve or the submanifold.

  In the case of $\cN\subset\{(N,\dots,N)\in\Z^k:N\in\Z_+\}$, in
  \cite{Shah:SLn} it was shown that for any analytic curve which is
  not contained in a proper affine subspace of $\R^k$, the DT cannot
  be improved along $\cN$ for almost all points on the curve; that is,
  for all $\mu<1$. In this article we will extend this result for any
  $\cN$.

  \begin{theorem}
    \label{thm:main-Dirichlet}
    Let $\cN$ be an infinite subset of $(\Z_+)^k$. Then for any
    analytic curve $\phi:[a,b]\to\R^k$ whose image is not contained in
    a proper affine subspace, the DT cannot be improved along $\cN$
    for $\phi(s)$ for Lebesgue almost every $s\in [a,b]$.
  \end{theorem}

  This theorem can be reformulated in terms of dynamics of flows on
  the homogeneous space $\SL(n,\R)/\SL(n,\Z)$;
  cf.~\cite[\S2.1]{Kleinbock+Weiss:Dirichlet}. We need to prove that
  certain sequence of expanding translates a curve on this space tend
  to become uniformly distributed. To adept the strategy of
  \cite{Shah:SLn} for the general $\cN$, we will need to overcome
  significant technical difficulties, whose resolution will require
  making new observations and developing much sharper methods.

  \subsection{Asymptotic equidistribution of translated curves}
  \label{subsec:notation1}
  Let $n\geq 2$ and $G=\SL(n,\R)$. For
  $\vtau=(\tau_1,\dots,\tau_{n-1})\in\R^{n-1}$ and
  $\vxi=(\xi_1,\dots,\xi_{n-1})\in \R^{n-1}$, define
  \begin{equation}
    \label{eq:not1}
    a_\vtau={\left[\begin{smallmatrix} e^{(\tau_1+\dots+\tau_{n-1})}\\ &
          e^{-\tau_1}\\ & & \ddots \\ & & & e^{-\tau_{n-1}}
        \end{smallmatrix}
      \right]}\quad \text{and} \quad
    u(\vxi)={\left[\begin{smallmatrix} 1 & \xi_1&\dots&\xi_{n-1}\\ & 1
          \\ & & \ddots \\ &&& 1
        \end{smallmatrix}
      \right]}.
  \end{equation}

  Let $\sT=\{\vtau_i\}_{i\in\N}\subset R^{n-1}$ be a sequence such
  that if $\vtau_i=(\tau_{i,1},\dots,\tau_{i,n-1})$ then
  $\tau_{i,1}\geq \tau_{i,2}\geq \dots\geq \tau_{i,n-1}\geq 0$, and
  $\norm{\vtau_i}\toinfty\infty$. After passing to a subsequence we
  further assume that there exists $1\leq m_1\leq n-1$ such that
  $\tau_{i,m_1}\toinfty\infty$ and
  $\lim_{i\to\infty}\tau_{i,r}=\tau(r)<\infty$ for $m_1<r\leq n-1$.

  For $2\leq m\leq n$, define
  \begin{equation}
    \label{eq:1}
    Q_m=\left\{\left[\begin{smallmatrix} g & \vw \\ \vzero & I_{n-m}
        \end{smallmatrix}\right]\in G: g\in \SL(m,\R),\,\vw\in {\rm
        M}_{m\times (n-m)}(\R)\right\},
  \end{equation}
  where $\vzero$ is the $(n-m)\times m$-zero matrix and $I_{n-m}$ is
  the $(n-m)\times (n-m)$-identity matrix. 

  The main goal of this article is to prove the following:

  \begin{theorem}
    \label{thm:main-action}
    Let $\phi:I=[a,b]\to\R^{n-1}$ be an analytic map whose image is
    not contained in a proper affine subspace. Let $L$ be a Lie group
    and $\Lambda$ a lattice in $L$. Let $\rho:G\to L$ be a continuous
    homomorphism. Let $x_0\in L/\Lambda$ and $H$ be a minimal closed
    subgroup of $L$ containing $\rho(Q_{m_1+1})$ such that the orbit
    $Hx_0$ is closed and admits a unique $H$-invariant probability
    measure, say $\mu_H$. Then for any bounded continuous function $f$
    on $L/\Lambda$ the following holds:
    \begin{equation}
      \label{eq:9L}
      \lim_{i\to\infty} \frac{1}{\abs{b-a}}\int_a^b
      f(\rho(a_{\vtau_i}u(\phi(s)))x_0)\,ds = 
      \int_{Hx_0} f(\rho(a_{\tau_0})x_0)\,d\mu_H(x),
    \end{equation}
    where $\vtau_0=(0,\dots,0,\tau(m_1+1),\dots,\tau(n-1))$.
  \end{theorem}

  Note that $\rho(Q_{m_1+1})$ is generated by $\Ad$-unipotent
  one-parameter subgroups of $L$. Hence by Ratner's
  theorem~\cite{R:uniform} $Hx_0$ is the closure of the
  $\rho(Q_{m_1+1})$-orbit of $x_0$.

  The above result in the case when
  $\vtau_i=(\tau_i,\dots,\tau_i)\in\R^{n-1}$ for a sequence
  $\tau_i\toinfty \infty$ was proved in \cite{Shah:SLn}. We will
  generalize that proof to obtain the above result. The main new
  contribution here is a strong general result about dynamics of
  intertwined linear actions of various $\SL(m,\R)$'s contained in
  $G$. Along with new interesting observations, its proof crucially
  uses the `Basic lemma' from \cite{Shah:SLn} on joint linear dynamics
  of various $\SL(2,\R)$'s cotained in $\SL(n,\R)$.

  For the basic application of the theorem we will put $L=G$,
  $\Lambda=\SL(n,\Z)$, $\rho$ the identity matrix, and
  $x_0=\SL(n,\Z)$. Then $H=Q_{m_1+1}$, because
  $Q_{m_1+1}\cap\SL(n,\Z)$ is a lattice in $Q_{m_1+1}$.

  For more examples, let $\sigma$ be an involutive automorphism
  of $\SL(n,\R)$ defined by
\begin{equation}
    \label{eq:30}
    \sigma(g):=\sw (\trn{g\inv}) \sw\inv,\quad\forall g\in \SL(n,\R),
  \end{equation}
where $\sw\in\GL(n,\R)$ permutes the standard basis
$\{e_1,\dots,e_n\}$ of $\R^n$ such that 
  \begin{equation}
    \label{eq:79}
    \sw(e_i)=e_{n+1-i},\quad \forall 1\leq i\leq n.
  \end{equation}
  
  Note that
  \begin{equation}
    \label{eq:31}
\begin{split}
    a'_{\vtau}&:=\sigma(a_\vtau)= \left[
      \begin{smallmatrix}
        e^{\tau_{n-1}}\\
        &\ddots \\
        &        & e^{\tau_{1}} \\
        & & & e^{-(\tau_1+\dots+\tau_{n-1})}
      \end{smallmatrix}
    \right],\\
    u'(\vxi)&:=\sigma(u(\vxi))= \left[
      \begin{smallmatrix}
        1 &      &      &\xi_{n-1}\\
        &\ddots&      &\vdots \\
        &      & 1    &\xi_{1}\\
        & & & 1
      \end{smallmatrix}
    \right] \quad\text{and}
  \end{split}
\end{equation}
\begin{equation*}
    Q'_m:=\sigma(Q_m)=\left\{ \left[
        \begin{smallmatrix}
          I_{n-m}              & \vw \\
          \vzero_{m\times (n-m)} & g
        \end{smallmatrix}
      \right]: g\in \SL(m,\R),\,\vw\in {\rm M}_{(n-m)\times
        m}(\R)\right\}.
  \end{equation*}
  
  Another example of Theorem~\ref{thm:main-action} is obtained as
  follows: Let $L=G\times G$ and define the homomorphism $\rho:G\to L$
  by
  \begin{equation}
    \label{eq:rho}
    \rho(g):=(g,\sigma(g)), \quad \forall g\in G.     
  \end{equation}
  Let $\Lambda=\SL(n,\Z)\times\SL(n,\Z)$. Then $\rho(\SL(n,\Z))\subset
  \Lambda$. Then $\rho(Q_{m_1+1})\cap\Lambda$ is a lattice in
  $\rho(Q_{m_1+1})$. Put $x_0=e\Lambda$. If we apply
  Theorem~\ref{thm:main-action} in this case then for its conclusion
  $H=\rho(Q_{m_1+1})$.

  \subsection{Some applications}

  Using the conclusion of the theorem in the above example we obtain the
  following result on non-improvability of Dirichlet's theorem on
  simultaneous Diophantine approximation in the dual form. 
  
  \begin{theorem}
    \label{thm:Dirichlet-multi-dual}
    Let $k\geq 1$, and $\phi:I=[a,b]\to \R^k$ be an analytic curve
    whose image is not contained in a proper affine subspace. Let
    $\cN$ be an infinite subset of $(\Z_+)^k$. Then for almost every
    $s\in I$ and any $\mu<1$, there exists infinitely many
    $(N_1,\dots,N_k)\in\cN$ such that both the following sets of
    inequalities are simultaneously insoluble:
    \begin{equation}
      \label{eq:32}
      \abs{q_1\phi(s)+\dots+q_{k}\phi(s)-p}\leq \mu(N_1\dots
      N_k)\inv,\quad
      \abs{q_i}\leq N_i \quad (1\leq i\leq k)
    \end{equation}
    for $(p,q_1,\dots,q_k)\in\Z^{k+1}\setminus\{\vzero\}$, and
    \begin{equation}
      \label{eq:36}
      \abs{q\phi_i(s)-p_i}\leq N_i\inv  \quad (1\leq i \leq k),\qquad
      \abs{q}\leq \mu N_1\dots N_k
    \end{equation}
    for $(q,p_1,\dots,p_k)\in\Z^{k+1}\setminus\{\vzero\}$.
  \end{theorem}

  Above statement is stronger than
  Theorem~\ref{thm:main-Dirichlet}. It also generalizes
  \cite[Theorem~1.4]{Shah:SLn}, which considered the case of $\cN$
  where $N_1=N_2=\dots=N_k$.

  \begin{corollary}
    \label{cor:manifold} 
    Let $\cN\subset(\Z_+)^n$ be an infinite set. Let $M$ be a
    connected Riemannian analytic submanifold of $\R^k$ which is not
    contained in a proper affine subspace of $\R^k$. Then with respect
    to the measure class on $M$ associated to the Riemannian volume
    form, for almost every $\vxi\in M$, the DT cannot be improved for
    $\vxi$ along $\cN$.

    In fact, the conclusion of Theorem~\ref{thm:Dirichlet-multi-dual}
    holds for almost all $\vxi\in M$ in place of $\phi(s)$.
  \end{corollary}

  It is interesting in the above results that we can take $\cN$ with
  bounded projections on some of the coordinates. In this case such
  results were not known earlier even for almost all points of $\R^k$.
  For Theorem~\ref{thm:main-Dirichlet}, it is essential that the
  limits $N_{0,j}$ for $m_1<j\leq n-1$ are integral:

\begin{theorem}
  \label{thm:counter}
  Let $\cN\subset\R^k$ be an unbounded sequence such that one of the
  coordinate converges to a non-integral real value. Then there exist
  $0<\mu<1$ such that for all but finitely many $(N_1,\dots, N_k)\in\cN$ and
  every $\vxi\in\R^k$, the system of inequalities \eqref{eq:100} admit
  nonzero integral solutions.
\end{theorem}

This fact, which is a consequence of Minkowski-Haj\'{o}s Theorem,
corroborates the counter examples given in
\cite[\S4.4]{Kleinbock+Weiss:Dirichlet} and answers a question raised
there.

\subsection{Non-improvability of Minkowski's theorem along certain
  $\cN$}

\begin{theorem}
  \label{thm:Minkowski-central}
  Let $\phi=(\phi_{ij}):I=[a,b]\to \SL(n,\R)$ be an analytic map such
  that $\R$-$\Span\{\phi_{1,j}(s):s\in I\}=\R^n$. Let $\sN$ be an
  infinite set of positive integers. Then for almost every $s\in I$,
  there exists an infinite subset $\sN_s\subset\sN$ such that for any
  $\mu<1$ the following system of inequalities is insoluble for any
  $\alpha\in\sN_s$ and $(x_1,\dots, x_n)\in\Z^n\setminus\{0\}$:
  \begin{align}
    \label{eq:88-Ma}
\begin{split}
    \abs{\phi_{1,1} x_1 +\dots+\phi_{1,n} x_n}&\leq \mu\alpha^{-n}\\
    \abs{\phi_{i,1} x_1+\dots+\phi_{i,n} x_{i,n}}&\leq \mu
    \alpha,\quad (2\leq i\leq n).
  \end{split}
\end{align}
\end{theorem}

\begin{theorem}
  \label{thm:weak-Minkowski}
  Let $\cN$ be an infinite subset of $(\Z_+)^{n-1}$ with unbounded
  projection on every coordinate. Let $\phi$ be as in
  Theorem~\ref{thm:Minkowski-central}. We further assume that
  $\phi_{i1}(s)\equiv 0$ and $\phi_{ij}(s)\equiv \phi_{ij}$ are
  constant functions for all $i\geq 2$. Then for almost every $s\in
  I$, there exists an infinite subset $\cN_s\subset \cN$ such that for
  every $\mu<1$ the following system of inequalities is insoluble in
  for any $(N_1,\dots,N_{n-1})\in\cN_s$ and
  $(x_1,\dots,x_n)\in\Z^n\setminus\{0\}$:
  \begin{align}
    \label{eq:88-WM}
    \begin{split}
    \abs{\phi_{11}(s) x_1 +\dots+\phi_{1n}(s) x_n}
    &\leq \mu(N_1\dots N_{n-1})\inv\\
    \abs{\phi_{i2} x_2\dots+\phi_{in} x_n} 
    &\leq \mu N_{i-1} \quad (2\leq i\leq n).
  \end{split}
\end{align}
\end{theorem}

It is a question whether the conditions on $\phi_{i,j}(s)$ can be
removed for $i>2$.

\subsection{Uniform versions of the equidistribution statement}
Let the notation be as in \S\ref{subsec:notation1}.
\begin{theorem}
  \label{thm:main-uniform}
  Let $\phi:I\to\R^{n-1}$ be an analytic curve whose image is not
  contained in a proper affine subspace. Let $L$ be a Lie group and
  $\Gamma$ be a lattice in $L$. Let $\rho:G\to L$ be continuous
  homomorphism. Let $x_0\in L/\Lambda$ be such that
  $\rho(Q_{m_1+1})x_0$ is dense in $L/\Lambda$. Let $x_i\toinfty x_0$
  be a convergent sequence in $L/\Lambda$. Then for any bounded
  continuous function $f$ on $L/\Lambda$
  \begin{equation}
    \label{eq:9-uniform}
    \lim_{i\to\infty} \frac{1}{\abs{b-a}}\int_a^b
    f(\rho(a_{\vtau_i}u(\phi(s)))x_i)\,ds=
    \int_{L/\Lambda} f\,d\mu_L,
  \end{equation}
  where $\mu_L$ is the unique $L$-invariant probability measure on
  $L/\Lambda$.
\end{theorem}

For the special case of $L=G$, $\rho$ the identity map, and $m_1=n-1$,
that is $Q_{m_1+1}=G$, we can take any convergent sequence $x_i\to
x_0$ in the above theorem.

A more general uniform version is as follows.

\begin{theorem}
  \label{thm:uniform:action}
  Let the notation be as in Theorem~\ref{thm:main-uniform}. Let
  $\cK$ be a compact subset of $L/\Lambda$. Then given $\epsilon>0$
  and a bounded continuous function $f$ on $L/\Lambda$, there exist
  finitely many proper closed subgroups $H_1,\dots,H_r$ of $L$ such
  that for each $1\leq i\leq r$, $H_i\cap\Lambda$ is a lattice in
  $H_i$ and there exists a compact set
  \begin{equation}
    \label{eq:59-unif}
    C_i\subset N(H_i,\rho(Q_{m_1+1})):=\{g\in L:\rho(Q_{m_1+1})g\subset gH_i\}    
  \end{equation}
  such that the following holds: Given any compact set
  \[
  F\subset \cK\setminus \cup_{i=1}^r C_i\Lambda/\Lambda
  \]
  there exists $i_0>0$ such that for any $x\in F$ and any $i\geq
  i_0$,
  \begin{equation}
    \label{eq:60-unif}
    \Abs{\frac{1}{b-a}\int_a^b f(\rho(a_{\vtau_i}u(\phi(s)))x)\,ds - 
      \int_{L/\Lambda} f\,d\mu_L} < \epsilon.
  \end{equation}
\end{theorem}

Both the above results in the special case, when for each $i$ all
coordinates of $\vtau_i$ are same, were obtained earlier in
\cite[\S1.2]{Shah:SLn}.

\subsubsection*{Acknowledgment} {\small I am very thankful to Shahar
  Mozes and Elon Lindenstrauss for several helpful discussions. I
  would also like to thank Dmitry Kleinbock for useful suggestions.}

\section{Deduction of Theorem~\ref{thm:Dirichlet-multi-dual} from
  Theorem~\ref{thm:main-action}}

By passing to a subsequence we express
\begin{equation}
  \label{eq:85}
  \cN=\{(N_{i,1},N_{i,2},\dots,N_{i,k})\in(\Z_+)^{k}:i\in\N\}.
\end{equation}
For proving Theorem~\ref{thm:Dirichlet-multi-dual} there is no loss of
generality if we apply a permutation of coordinates on
$(\Z_+)^{k}$. Therefore, since there are only finitely many coordinate
permutations, after passing to a subsequence, we may assume that
\begin{equation}
  \label{eq:86}
  N_{i,1}\geq N_{i,2}\geq\dots\geq N_{i,k}, \quad \forall i\in\N.  
\end{equation}
Since $\cN$ is infinite, there exists $m_1\geq 1$ such that after by
passing to a subsequence, we may further assume that
$N_{i,m_1}\toinfty\infty$, and for each $m_1<j\leq k$ there exist
$N_{0,j}\in\Z_+$ such that $N_{i,j}=N_{0,j}$ for all $i$. We define
\[
\vtau_i=(\log N_{i,1},\dots,\log N_{i,{k}})\in (\R_{\geq0})^{k}, \quad
\forall i\in \N,
\]
and put $\sT=(\vtau_i)_{i\in\N}$. Then $\sT$ satisfies the conditions
of \S\ref{subsec:notation1}. We put
\begin{equation}
  \label{eq:92}
  \vtau_0:=(0,\dots,0,\log N_{0,{m_1+1}},\dots,\log N_{0,k}).
\end{equation}

Let $n=k+1$. We identify the space $\Omega$ of unimodular lattices in
$\R^n$ with $\SL(n,\R)/\SL(n,\Z)$. Given $0<\mu<1$, we define
\begin{align}
\label{eq:73}
\begin{split}
  B_\mu&=\{(\xi_1,\dots,\xi_n)\in\R^n:\sup_{1\leq i\leq
    n}\abs{\xi_i}\leq 1\}\\
  K_\mu&=\{\Delta\in\Omega: \Delta\cap B_\mu=\{0\}\}.
\end{split}
\end{align}

For $(N_1,\dots,N_k)\in\cN$, let $\vtau=(\log N_1,\dots,\log
N_k)$. Then for any $s\in I$, $\vx=(p,q_1,\dots,q_k)\in\Z^n$ and
$\vx'=(p_k,\dots,p_1,q)\in\Z^k$, we have
\begin{equation}
  \label{eq:53-tau}
  \rho(a_{\vtau}(\phi(s)))(\vx,\vx')=
  \left(
    \left(\begin{smallmatrix}
        (N_1\dots N_k)(p+\sum_{i=1}^k q_i\phi_i(s))\\
        N_1\inv q_1\\
        \vdots\\
        N_k\inv q_k
      \end{smallmatrix}
    \right),
    \left(\begin{smallmatrix}
        N_k(q\phi_k(s)+p_k)\\
        \vdots\\
        N_1(q\phi_1(s)+p_1)\\
        (N_1\dots N_k)\inv q
      \end{smallmatrix}
    \right)
  \right).
\end{equation}

Therefore (cf.~\cite[\S2.1]{Kleinbock+Weiss:Dirichlet},
\cite[\S2]{Shah:SLn})
\begin{align}
  \label{eq:94}
  \begin{split}
  \text{\eqref{eq:32} is soluble} &\iff a_\vtau u(\phi(s))\vx\in
  B_{\mu}\iff a_{\vtau}u(\phi(s))\Z^n\not\in
  K_\mu\\
  \text{\eqref{eq:36} is soluble} &\iff a'_{\vtau}u'(\phi(s))\vx'\in
  B_\mu \iff a'_{\tau}u'(\phi(s))\Z^n\not\in K_\mu.
\end{split}
\end{align}

Hence Theorem~\ref{thm:Dirichlet-multi-dual} will be proved if we
prove the following:

\begin{proposition}
  \label{prop:E}
  Put $x_0=:(\Z^n,\Z^n)\in\Omega\times\Omega$. Given $\mu<1$, define
  \begin{equation}
    \label{eq:87}
    E_\mu:=\{s\in [a,b]:\rho(a_{\vtau_i}u(\phi(s)))x_0\not\in
    K_\mu\times K_\mu\text{ for all large $i$}\}.
  \end{equation}
  Then $\abs{E_\mu}=0$, where $\abs{\cdot}$ denotes the Lebesgue
  measure.
\end{proposition}

In order to deduce this proposition from Theorem~\ref{thm:main-action}
we need the following result, especially when $m_1<n-1=k$.

We define a compact set
\begin{align}
  \label{eq:81}
  K_1=\cap_{0<\mu<1} K_\mu
  =\{\Delta\in\Omega: \sup_{1\leq
      i\leq n}\abs{\xi_i}\geq 1, \forall
    (\xi_1,\dots,\xi_n)\in\Delta\setminus\{0\}\} 
\end{align}

\begin{theorem}
  \label{thm:intersect}
  $\rho(a_{\vtau_0} Q_{m_1+1})x_0\cap (K_1\times K_1)\neq\emptyset$.
\end{theorem}

To prove this we will need the easier inclusion part of the following
fact, which was guessed by Minkowski (1896) and proved by Haj\'{o}s in
1941, see \cite[XI.1.3]{Cassels:Geometry of Numbers}. Its full
strength will be used later for proving Theorem~\ref{thm:counter}.
 
\begin{theorem}
  \label{thm:K1}
  Let $N$ denotes the group of upper triangular unipotent matrices in
  $\SL(n,\R)$. Let $\cW_n=\{w\in \GL(n,\Z)$: $w$ permutes the standard
  basis of $\R^n\}$. Then
  \begin{equation}
    \label{eq:76}
    K_1=\bigcup_{w\in \cW_n} (wNw\inv) \Z^n.
  \end{equation}
  \qed
\end{theorem}

\begin{proposition}
  \label{prop:tau-Q}
  Let $N^-$ be the lower triangular unipotent subgroup of $\SL(n,\R)$
  and $\Gamma=\SL(n,\Z)$. Let $(N_1,\dots,N_{n-1})\in(\Z_+)^{n-1}$ and
  $\vtau:=(\log N_1,\dots,\log N_{n-1})$. Then
  \begin{equation}
    \label{eq:80}
    a_\vtau Q_1 \Gamma \cap N^- \Gamma\neq \emptyset.
  \end{equation}
\end{proposition}

\begin{proof}
  For the involutive automorphism $\sigma$ as defined by
  \eqref{eq:30}, we have $\sigma(N^-)=N^-$, $\sigma(\Gamma)=\Gamma$,
  $\sigma(a_\vtau)=a'_{\vtau}$ and $\sigma(Q_{1})=Q'_1$. Therefore it
  is enough to prove that
  \begin{equation}
    \label{eq:82}
    a'_{\vtau}Q'_{1}\cap N^-\Gamma\neq \emptyset.    
  \end{equation}

  We write $k=n-1$. Note that
  \begin{equation}
    \label{eq:84}
    a'_{\vtau}Q'_{1}=
    \left\{
      \left[
        \begin{matrix}
          N_k &        &     & \xi_k\\
          & \ddots &     & \vdots \\
          &        & N_1 & \xi_1 \\
          & & & (N_1\dots N_{k})\inv
        \end{matrix}
      \right]: \xi_i\in\R \right\}.
  \end{equation}

  We define
  \begin{equation}
    \label{eq:83}
    \gamma = 
    \left[
      \begin{matrix}
        N_{k}    &        &        &         & 1 \\
        N_{k}-1  & N_{k-1}    &        &      & 1 \\
        N_{k}-1  & N_{k-1}-1  & \ddots &     & \vdots\\
        \vdots & \vdots & \vdots & N_1         & 1 \\
        N_{k}-1 & N_{k-1}-1 & \dots & N_1-1 & 1
      \end{matrix}
    \right]\in\SL(n,\Z).
  \end{equation}
  We express the matrix $\gamma$ as a sum of a lower triangular matrix
  $L$ of determinant $1$, and a matrix $M$ with only the last column
  nonzero. Then we choose $h\in N^-$ such that
  $hL=\diag(N_1,\dots,N_k,(N_1\dots N_k)\inv)$. Also $hM$ is a matrix
  with only the last column nonzero. In particular, $h\gamma=(hL+hM)$
  is an upper triangular matrix of determinant $1$. Hence $h\gamma\in
  a'_\vtau Q'_1$ by \eqref{eq:84}. This proves \eqref{eq:82}.
\end{proof}

\begin{proof}[Proof of Theorem~\ref{thm:intersect}]
  Since $Q_1\subset Q_{m_1+1}$, by Proposition~\ref{prop:tau-Q} there
  exist $g\in Q_{m_1+1}$, $h\in N^-$ and $\gamma\in\Gamma$ such that
  $a_{\vtau_0}g=h\gamma$.  Therefore
  $\rho(a_{\vtau_0}g)=\rho(h)\rho(\gamma)$. Now $\rho(\gamma)x_0=x_0$
  and $\sw N \sw\inv =N^-$ for $\sw\in \cW_n$ as in
  \eqref{eq:79}. Therefore $\rho(a_{\vtau_0}g)x_0=\rho(h)x_0\in
  K_1\times K_1$ by Theorem~\ref{thm:K1}.
\end{proof}

\begin{proof}[Proof of Proposition~\ref{prop:E}]
  Let $\rho$ be as in \eqref{eq:rho}. Then the orbit
  $\rho(Q_{m_1+1})x_0$ is closed and admits a unique
  $\rho(Q_{m_1+1})$-invariant probability measure, say $\lambda$. Fix
  $\mu<1$. Then $K_\mu\times K_\mu$ contains an open neighbourhood of
  $K_1\times K_1$. Therefore by Theorem~\ref{thm:intersect},
  $K_\mu\times K_\mu$ contains a nonempty open subset of
  $\rho(a_{\vtau_0}Q_{m_1+1})x_0$. Therefore there exists $\epsilon>0$
  such that
  \begin{equation}
    \label{eq:89}
    \lambda(\rho(a_{\vtau_0})\inv(K_\mu\times K_\mu))>\epsilon.
  \end{equation}
  So there exists $f\in\Cc(L/\Lambda)$ such that $0\leq f\leq 1$,
  $\supp(f)\subset K_\mu\times K_\mu$ and
  \begin{equation}
    \int_{x\in \rho(Q_{m_1+1})x_0} f(\rho(a_{\vtau_0})x)\,d\lambda(x)\geq \epsilon/2. 
  \end{equation}

  Let $J$ be any subinterval of $[a,b]$ with nonempty interior. Then
  by Theorem~\ref{thm:main-action} there exists $i_0\in\N$ such that
  for all $i\geq i_0$, we have
  \begin{equation}
    \label{eq:91}  
    \frac{1}{\abs{J}}\int_J f(\rho(a_{\vtau_i}u(\phi(s)))x_0)\,ds \geq
    \int 
    f(\rho(a_{\vtau_0})x)\,d\lambda(x)-\epsilon/4\geq \epsilon/4.
  \end{equation}
  Therefore by combining the definition of $E_\mu$ and the choice of
  $f$ we conclude that $\abs{E\cap J}\leq
  (1-\epsilon/4)\abs{J}$. Since $J$ is an arbitrary open subinterval
  of $I$, by the Lebesgue density theorem $\abs{E_\mu}=0$.
\end{proof}

Thus we have completed the deduction of
Theorem~\ref{thm:Dirichlet-multi-dual}. \qed

\begin{proof}[Proof of Corollary~\ref{cor:manifold}]
  We note that $M$ can be measurably fibered by analytic curves such
  that almost every curve in the fiber is not contained in a proper
  affine subspace of $\R^k$. Therefore the conclusions of
  Theorem~\ref{thm:main-Dirichlet} and
  Theorem~\ref{thm:Dirichlet-multi-dual} hold for almost all points on
  each of these curves. Therefore by Fubini's theorem, the conclusion
  of the corollary follows.
\end{proof}

\begin{proof}[Proof of Theorem~\ref{thm:Minkowski-central}]
  The result follows by the arguments as above. We need to consider
  only the first factor in \eqref{eq:53-tau} and
  $Q_{m_1+1}=G$. Clearly $K_\mu$ has strictly positive measure on
  $G/\Gamma$. The only difference is that to conclude \eqref{eq:91} we
  need to use \cite[Theorem~1.8]{Shah:SLn}, where the same conclusion
  as that of Theorem~\ref{thm:main-action} was obtained for the given
  map $\phi$ and $\vtau_i=(\tau_i,\dots,\tau_i)\in\R^{n-1}$ for each
  $i$ such that $\tau_i\toinfty\infty$.
\end{proof}

\begin{proof}[Proof of Theorem~\ref{thm:weak-Minkowski}]
  Again the deduction of this result is as above, one only considers
  the first factor. Here by our choice $Q_{m_1+1}=G$ and $K_\mu$ has
  strictly positive measure on $G/\Gamma$. Now note that there exists
  $g_0\in G$ and an analytic curve $\psi:I\to \R^{n-1}$ such that
  $\phi(s)=u(\psi(s))g_0$ for all $s\in I$. By our condition, the
  image of $\psi$ is not contained in any proper affine subspace of
  $\R^{n-1}$. We then apply Theorem~\ref{thm:main-action} for
  $x_0=g_0\Z^n$.
\end{proof}

\begin{proof}[Proof of Theorem~\ref{thm:counter}]
  Let the notation be as in the beginning of this section. Without loss
  of generality we may assume that \eqref{eq:86} holds. The only
  difference is that now the $N_{i,j}$'s are real numbers instead of
  integers. We may also assume $N_{i,k}\toinfty N_{0,k}$ and
  $N_{0,k}\not\in\Z$. Let $n=k+1$ and $x_0=\Z^n\in\Omega$. Then in
  view of \eqref{eq:53-tau}, considering only the first factor, it is
  enough to show that there exist $0<\mu<1$ and $i_0\in\N$ such that
  \begin{equation}
    \label{eq:101}
    a_{\vtau_i}Q_{m_1+1}x_0\cap \cK_\mu=\emptyset,\quad\forall i\geq i_0.   
  \end{equation} 
  Since $\cK_1=\bigcap_{0<\mu<1}\cK_\mu$, it is enough to show that
  \begin{equation}
    \label{eq:103}
    a_{\vtau_0}Q_{m_1+1}x_0\cap \cK_1=\emptyset.    
  \end{equation}

  Suppose if this intersection is nonempty, then by
  Theorem~\ref{thm:K1} due to Minkowski and Haj\'os,
  \begin{equation}
    \label{eq:104}
    a_{\vtau_0} Q_{m_1+1} \cap \cW_n N\Gamma\neq \emptyset.  
  \end{equation}
  Applying $\sigma$ on both sides, we get
  \begin{equation}
    \label{eq:107}
    a'_{\vtau_0} Q'_{m_1+1} \cap \cW_n N \Gamma\neq\emptyset.
  \end{equation}

  Let $\{e_1,\dots,e_n\}$ denote the standard basis of $\R^n$. Since
  $Q'_{m_1+1} e_1=e_1$ and $a_{\vtau_0}e_1=N_{0,k}$, there exist
  $\vv=q_1e_1+\dots+q_ne_n\in \Gamma e_1\subset \Z^n$, $w\in\cW_n$ and
  $g=(x_{ij})\in N$ such that
  \begin{equation}
    \label{eq:105}
    N_{0,{n-1}} e_1 = w g \vv.
  \end{equation}
  Let $1\leq r \leq n$ be such that $w\inv e_1=e_r$. Then we have
  \begin{align}
    \label{eq:109}
    N_{0,{n-1}}&=q_r+ \sum_{j>r} x_{ij} q_j \\
    \label{eq:109b}
    0&= q_i + \sum_{j>i} x_{ij} q_j \quad (i>r).
  \end{align}
  Putting $i=n$ in \eqref{eq:109b} we get that $q_n=0$. Now for any
  $i_0>r$, if $q_j=0$ for all $j>i_0$ then putting $i=i_0$ in
  \eqref{eq:109b} we get that $q_{i_0}=0$. Therefore by induction
  $q_j=0$ for all $j>r$. Therefore \eqref{eq:109} becomes
  $N_{0,n-1}=q_r\in\Z$, a contradiction. This shows that the
  intersection in \eqref{eq:103} cannot be nonempty, and the proof is
  complete.
\end{proof}

Now we begin the proof of the main Theorem~\ref{thm:main-action}.

\section{Nondivergence of translates}
Let the notation be as in \S\ref{subsec:notation1} and the statement
of Theorem~\ref{thm:main-action}. We consider the action of $G$ on
$L/\Lambda$ via the homomorphism $\rho$; that is, for any $x\in
L/\Lambda$ and $g\in G$, we have $gx:=\rho(g)x$. Let
$\{x_i\}_{i\in\N}$ be a sequence in $L/\Lambda$ such that $x_i\toinfty
x_0$. For any $i\in\N$ define $\mu_i$ to be the probability measure on
$L/\Lambda$ as
\begin{equation}
  \label{eq:7}
  \int_{L/\Lambda} f\,d\mu_i:=\frac{1}{\abs{I}}\int_I
  f(a_{\vtau_i}u(\phi(s))x_i)\,ds, \quad \forall f\in\Cc(L/\Lambda).
\end{equation}

\begin{theorem}
  \label{thm:nondiv}
  Given $\epsilon>0$ there exists a compact set $\cF\subset L/\Lambda$
  such that $\mu_i(\cF)\geq 1-\epsilon$ for all large $i\in\N$.
\end{theorem}

In the case of $L=\SL(n,\R)$, $\rho$ the identity map, and
$\Lambda=\SL(n,\Z)$, the result was obtained by Kleinbock and
Margulis~\cite{Klein+Mar:Annals98}.

\subsection{} \label{subsec:sH} Let $\sH$ denote the collection of
analytic subgroups $H$ of $G$ such that $H\cap\Lambda$ is a lattice in
$H$ and a unipotent one-parameter subgroup of $H$ acts ergodically
with respect to the $H$-invariant probability measure on
$H/H\cap\Lambda$.  Then $\sH$ is a countable collection
\cite{R:measure,Shah:uniform}.

Let $\la{L}$ denote the Lie algebra associated to $L$. Let
$V=\oplus_{d=1}^{\dim\la{L}}\wedge^d\la{L}$ and consider the
$(\oplus_{d=1}^{\dim{\la{l}}}\wedge^d\Ad)$-action of $L$ on $V$. Given
$H\in\sH$, let $\la{h}$ denote its Lie algebra, and fix
$p_H\in\wedge^{\dim\la{h}}\la{h}\setminus \{0\}\subset V$. Let
$\norL(H)$ denote the normalizer of $H$ in $L$. Then
\begin{equation}
  \label{eq:stab}
  \Stab_L(p_H)=\noroneL(H):=\{g\in \norL(H):\det((\Ad g)|_{\la{H}})=1\}.
\end{equation}

 \begin{proposition}[\cite{Dani+Mar:limit}]
  \label{prop:discrete}
  The orbit $\Lambda \cdot p_H$ is a discrete subset of $V$.  \qed
\end{proposition}

Consider the $G$ action on $V$ via $\rho$; that is $gv=\rho(g)v$ for
all $g\in G$ and $v\in V$. Given a sequence $\sT$ as in
\S\ref{subsec:notation1}, we define
\begin{equation}
  \label{eq:68}
\begin{split}
  V^-_{\sT}&=\{v\in V:a_{\vtau_i}v\toinfty0\}, \quad V^+_{\sT}=\{v\in V:a_{\vtau_i}\inv v\toinfty0\}\\
  V^0_{\sT}&=\{\in V: a_{\vtau_i}v\toinfty v_1 \text{ and }
  a_{\vtau_i}\inv v\toinfty v_2 \text{ for some $v_1,v_2\in V$}\}.
\end{split}
\end{equation}
Since $\{a_{\vtau}:\vtau\in\R^{n-1}\}$ acts on $V$ by
$\R$-diagonalizable commuting automorphisms, by passing to a
subsequence of $\sT$, we have 
\begin{equation}
  \label{eq:117}
V=V^-_{\sT}\oplus V^0_{\sT}\oplus V^+_{\sT}.   
\end{equation}
Let $\pi_0^{\sT}:V\to V^0_{\sT}$ denote the corresponding
projection.

\subsection{Margulis-Dani non-diverergence criterion}

We recall the following criterion based on
\cite{Dani+Mar:asymptotic,Shah:horo,Klein+Mar:Annals98}. Here we use
the information that $\phi$ is an analytic map. 

\newcommand{\citeone}{\cite[Prop.3.4]{Shah:SLn}}
\begin{proposition}[\citeone]
  \label{prop:return}
  There exists a finite collection $\sW\subset \sH$ (depending only on
  $L$ and $\Lambda$) such that the following holds: Given $\epsilon>0$
  and $R>0$, there exists a compact set $\cF \subset L/\Lambda$ such
  that for any $h_1,h_2\in L$ and a subinterval $J\subset I$, one of
  the following conditions is satisfied:
  \begin{enumerate}
  \item[I)] There exists $\gamma \in \Lambda$ and $W\in\sW$ such that
    \[
    \sup_{s \in J} \norm{h_1u(\phi(s))h_2 p_W}< R.\]
  \item[II)] $\frac{1}{\abs{J}}\abs{\{s \in J :
      h_1u(\phi(s))h_2\Lambda/\Lambda \in \cF \}} \geq 1- \epsilon$.
  \end{enumerate}
\end{proposition}

\begin{proof}[Proof of Theorem~\ref{thm:nondiv}]
  Let $g_0\in L$ such that $x_0=g_0\Lambda$. Take a sequence
  $R_k\stackrel{k\to\infty}{\longrightarrow}0$ of positive reals.
  Suppose that Theorem~\ref{thm:nondiv} fails to hold for some
  $\epsilon>0$. Then for any $R>0$ and any compact set $\cF$ and
  infinitely many $i\in\N$, the condition~(II) of
  Proposition~\ref{prop:return} fails to hold for $J=I$,
  $h_1=a_{\vtau_i}$ and $h_2=g_0$; and hence the condition~(I)
  holds. Therefore after passing to subsequences, there exists
  $W\in\sH$, and for each $i$ there exists $\gamma_i\in\Lambda$ such
  that
  \begin{equation}
    \label{eq:65}
    \sup_{s\in I} \norm{a_{\vtau_i}u(\phi(s))g_0\gamma_ip_W}\leq
    R_i\toinfty 0.
  \end{equation}
  By Proposition~\ref{prop:discrete}, there exists $r_0>0$ such that
  $\norm{g_0\gamma_ip_W}\geq r_0$ for each $i$. We put
  $v_i=g_0\gamma_ip_W/\norm{g_0\gamma_i p_W}$.  Then
  \begin{equation}
    \label{eq:66}
    \sup_{s\in I}\norm{a_{\vtau_i}u(\phi(s))v_i}
    \leq R_i/\norm{g_0\gamma_ip_W}\leq R_i/r_0\toinfty 0.
  \end{equation}
  After passing to a subsequence $v_i\to v\in V$ and $\norm{v}=1$.
  Now from \eqref{eq:68}, \eqref{eq:117} and \eqref{eq:66} we deduce
  that
  \begin{equation}
    \label{eq:minus}
    u(\phi(s))v\in V^-_{\sT}, \quad\forall s\in I.
  \end{equation}
  
  A main result on linear dynamics proved in the next section, namely
  Corollary~\ref{cor:rep2-main}, in view of Lemma~\ref{lema:sT} states
  that: {\em For any finite dimensional linear representation $V$ of
    $G$, any $v\in V\setminus\{0\}$ and any $\cB\subset \R^{n-1}$ not
    contained in a proper affine subspace of $\R^{n-1}$,
  \begin{equation}
    \label{eq:zerominus}
   \text{if $u(e)v\in V^-_{\sT}+V^0_{\sT}$ for all $e\in\cB$ then
     $\pi_0^{\sT}(u(e)v)\neq 0$ for all $e\in\cB$}.
  \end{equation}}

Since $\{\phi(s):s\in I\}$ is not contained in a proper affine
subspace of $\R^{n-1}$, \eqref{eq:minus} implies the {\em if\/}
condition of \eqref{eq:zerominus} but contradicts its implication.
\end{proof}

From Theorem~\ref{thm:nondiv} we deduce the following:

\begin{corollary} \label{cor:mu-return} After passing to a
  subsequence, $\mu_i\to \mu$ as $i\to\infty$ in the space of
  probability measures on $\lml$ with respect to the weak-$\ast$
  topology; that is,
  \begin{equation}
    \label{eq:2}
    \int_{\lml} f\,d\mu_i\toinfty\int_{\lml}
    f\,d\mu,\quad\forall f\in\Cc(\lml).
  \end{equation}
\end{corollary}

\section{Dynamics of the intertwined linear actions of various
  $\SL(m,\R)'s$ contained in $G$}

\label{sec:dynamics}

In this section will give proofs of the new technical results of this
article, including the one used above. We will also derive their
further consequences, which will be crucially needed for applying the
linearization techniques in combination with Ratner's theorem, in
order to describe the limit measure $\mu$ in later sections.

\subsection{Layered presentation for the infinite sequence $\sT$}
\label{subsec:cT}
Let $\sT=(\vtau_i)_{i\in\N}$ be an unbounded sequence as in
\S\ref{subsec:notation1}; that is,
$\vtau_i=(\tau_{i,1},\dots,\tau_{i,n})\in\R^{n-1}$ such that
$\tau_{i,1}\geq \tau_{i,2}\geq \dots\geq \tau_{i,n}\geq 0$. By passing
to a subsequence we may further assume that there exist
$k\in\{1,\dots,n\}$ and integers $n-1>m_1>m_2>\dots>m_k\geq 1$ such
that the following holds:
\begin{align*}
  \lim_{i\to\infty} \tau_{i,r} &<\infty \quad (m_1<r\leq n-1)\\
  \lim_{i\to\infty} \tau_{i,m_1}&=\infty\\
  \lim_{i\to\infty}
  (\tau_{i,m_{\ell+1}}-\tau_{i,m_\ell})&=\infty \quad
  (1\leq \ell\leq k-1)\\
  \lim_{i\to\infty}
  (\tau_{i,r}-\tau_{i,m_\ell})&<\infty\quad (m_{\ell+1}<r\leq m_\ell,\
  1\leq \ell\leq k), 
\end{align*}
where $m_{k+1}=0$. Define 
\begin{align*}
\bar\tau_{i,r}&=0\quad (m_1<r\leq n-1)\\
\bar\tau_{i,r}&=\tau_{i,m_\ell}\quad (m_{\ell+1}<r\leq m_\ell,\ 1\leq
\ell \leq k)\\
\bar\vtau_i&=(\bar\tau_{i,1},\dots,\bar\tau_{i,n-1}).
\intertext{Then $\lim_{i\to\infty} (\vtau_i-\bar\vtau_i)$ exists in $\R^{n-1}$. Define}
t_{i,1}&=\bar\tau_{i,m_1}\\
t_{i,\ell}&=\bar\tau_{i,m_{\ell}}-\bar\tau_{i,m_{\ell-1}} \quad (2\leq \ell\leq k)\\
\vt_i&=(t_{i,1},\dots,t_{i,k})\in(\R_+)^k.  
\end{align*}
Thus we obtain a sequence $\cT=(\vt_i)_{i\in\N}$ associated to the
given sequence $\sT$ as above. Now $t_{i,\ell}\toinfty\infty$ ($1\leq
\ell \leq k$).

\subsection{Notation and set up}
\label{subsec:notation2}
We are given natural numbers $n$, $k<n$ and $n-1\geq m_1>\dots>m_k\geq
1$. Let $E_{i,j}$ denote the $n\times n$-matrix with $1$ in the
$(i,j)$-th coordinate and $0$ in the rest. For $1\leq \ell \leq k$, we
define
\begin{equation}
  \label{eq:13}
  \cA_\ell=m_\ell E_{1,1} - \sum_{j=2}^{m_\ell+1}E_{j,j}.
\end{equation}

Given $1\leq \ell\leq k$, let
$\cT=(\vt_i)_{i=1}^\infty\subset(\R_+)^\ell$ be a sequence such that
each coordinate of $\vt_i$ tends to infinity as $i\to\infty$.  For
$\vt=(t_1,\dots,t_\ell)\in\R^\ell$, define
\begin{equation}
  \label{eq:23}
  \cA(\vt):=t_1\cA_1+\dots+t_\ell\cA_\ell.
\end{equation}

Note that for the sequences $\sT$ and $\cT$ as described in
\S\ref{subsec:cT}, we have
\begin{equation}
  \label{eq:a_tau}
  a_{\bar\vtau_i}=\exp(\cA(\vt_i)),\quad \forall i\in\N.
\end{equation}

Let $V$ be a finite dimensional linear representation of $\SL(n,\R)$.

For $\vmu=(\mu_1,\dots,\mu_\ell)\in\R^\ell$, we define
\begin{equation}
  \label{eq:33}
  V_\vmu = \{v\in V: \text{$\cA_i v=\mu_i v$ for $1\leq i\leq \ell$}\}.
\end{equation}
Thus if $v\in V_{\vmu}$ then $\cA(\vt)v=(\vmu\cdot\vt)v$.  The set
$\Delta_\ell:=\{\vnu\in\R^\ell:V_\vnu\neq 0\}$ is finite, and
\begin{equation}
  \label{eq:51}
  V=\oplus_{\vnu\in\Delta_\ell} V_\vnu.
\end{equation}
Let $\pi_\vmu^{\cT}:V\to V_\vmu$ denote the corresponding projection. We put
\begin{equation}
  \label{eq:98}
  \Delta:=\Delta_k.
\end{equation}

We define
\begin{align}
\label{eq:26}
\begin{split}
  V^{-}(\cT)&=\{v\in V:\exp(\cA(\vt_i))v\toinfty 0\}\\
  & = \sum_{\{\vmu\in\Delta_\ell:\lim_{i\to\infty}\vmu\cdot\vt_i= -\infty\}} V_\vmu\\
  V^+(\cT)&=\{v\in V:\exp(-\cA(\vt_i))v\toinfty 0\} \\
  &=\sum_{\{\vmu\in\Delta_\ell:\lim_{i\to\infty} \vmu\cdot\vt_i=\infty\}} V_\mu\\
  V^0(\cT)&=\Bigl\{v\in V:
  \begin{array}{l}
    \text{both $\exp(-\cA(\vt_i))v$ and $\exp(\cA(\vt_i))v$} \\
    \text{converge in $V$ as $i\to\infty$}
  \end{array}
\Bigr\}\\
&=\sum_{\{\vmu\in\Delta_\ell:
  \abs{\lim_{i\to\infty}\vmu\cdot\vt_i}<\infty\}} V_\mu.
\end{split}
\end{align}
Then after passing to a subsequence of $\cT$, we have
\begin{equation}
  \label{eq:27b}
  V=V^+(\cT)\oplus V^0(\cT) \oplus V^-(\cT).
\end{equation}
Let $\pi_0^{\cT}:V\to V^0(\cT)$ denote the corresponding projection.

In view of \eqref{eq:a_tau} we have the following:

\begin{lemma}
  \label{lema:sT}
  Let $\sT$ and $\cT$ be as defined in \S\ref{subsec:cT}. Let
  $V^-_{\sT}$, $V^+_{\sT}$ and $V^0_{\sT}$ be as defined in
  \eqref{eq:68}. Then
  \begin{equation*}
    V^-_{\sT}=V^+(\cT),\quad V^0_{\sT}=V^0(\cT)\quad \text{and}\quad 
    V^+_{\sT}=V^0(\cT).
  \end{equation*}
  In particular, the corresponding projections
  $\pi_0^{\sT}=\pi_0^{\cT}$. \qed
\end{lemma}

\subsection{Main Result} \label{subsubsec:cTk} In the following
discussion, let $\cT=(\vt_i)_{i=1}^\infty\subset(\R_+)^k$, where
$k\geq 2$ is a sequence such that (each coordinate of
$\vt_i$)$\toinfty\infty$ and $V=V^0(\cT)\oplus V^+(\cT)\oplus
V^-(\cT)$. For $\vt=(t_1,\dots,t_k)\in\R^k$, let
$\vt'=(t_1,\dots,t_{k-1})\in\R^{k-1}$.  Given $\cT$ as above, let
\begin{equation}
  \label{eq:114}
  \cT'=(\vt_i')_{i=1}^\infty\subset (\R_+)^{k-1}.  
\end{equation}

\begin{proposition}[Basic Lemma-II]
  \label{prop:rep2}
  Let $\cB$ be an affine basis of $\R^{n-1}$; that is,
  $\{e'-e:e'\in\cB\}$ is a basis of $\R^{n-1}$ for any
  $e\in\cB$. Suppose that $v\in V$ is such that
  \begin{equation}
    \label{eq:29}
    u(e)v\in V^0(\cT)+V^-(\cT),\quad \forall\, e\in\cB.
  \end{equation}
  Then for any $e\in\cB$,
  \begin{equation}
    \label{eq:95}
    u(e)v\in V^0(\cT')+V^-(\cT').
  \end{equation}

  Moreover for any $e\in\cB$,
\begin{align}
 \label{eq:120b}
\text{if $\pi_{\vzero'}^{\cT'}(u(e)v)\neq 0$ then
  $\pi_{\vzero}^{\cT}(\pi(u(e)v))\neq 0$.}
\end{align}
\end{proposition}

The rest of \S\ref{subsubsec:cTk} is devoted to the proof of this
proposition.
  
\subsubsection{Notation}
Let $\la{c}=\Span\{\cA_1,\dots,\cA_k\}$. Let $\la{h}$ denote the Lie
algebra generated by elements
\begin{equation}
  \label{eq:17}
  E_{1,1},\, E_{1,j},\, E_{j,1},\, E_{j,j}, \quad \text{where
    $2\leq j\leq m_k+1$}.
\end{equation}
Then $\la{h}$ is naturally isomorphic to $\la{SL}(m_k+1,\R)$. Since
$[\la{c},\la{h}]=\la{h}$, we have that $\la{c}+\la{h}$ is a reductive
subalgebra isomorphic to $\la{z}_{\la{c}}(\la{h})\oplus \la{h}$, where
$\la{z}_{\la{c}}(\la{h})$ denotes the centralizer of $\la{h}$ in
$\la{c}$.

We define a preorder $\preceq$ on $\Delta$ by
\begin{equation}
  \label{eq:34}
  \vmu\preceq \vnu \iff (\text{$\vmu\cdot\vt \leq \vnu\cdot \vt$ for
    all $\vt\in\cT$}).
\end{equation}
By passing to a subsequence of $\cT$ we may assume that $\preceq$ is a
total preorder. Note that ($\vmu\preceq\vnu$ and $\vnu\preceq\vmu$)
does not imply $\vmu=\vnu$.

For $\vmu=(\mu_1,\dots,\mu_k)\in\R^k$ we define
$\vmu'=(\mu_1,\dots,\mu_{k-1})\in\R^{k-1}$.

\begin{lemma}[Positivity Lemma]
  \label{lema:order}
  Let $W$ be an irreducible $(\la{c}+\la{h})$-submodule of $V$. Let
  \begin{equation}
    \label{eq:97}
    \Delta(W):=\{\vmu\in\Delta:\pi_\vmu ^{\cT}(W)\neq 0\}.
  \end{equation}
  Then for any $\vmu,\vnu\in \Delta(W)$ and $\vt\in(\R_{>0})^k$,
  where $k\geq 2$, we have
  \begin{equation}
    \label{eq:39}
    \vmu\cdot\vt \geq \vnu\cdot\vt \Longleftrightarrow \mu_k \geq \nu_k
    \Longleftrightarrow \vmu'\cdot\vt' \geq \vnu'\cdot\vt', 
  \end{equation}
  where $\vmu=(\mu_1,\dots,\mu_k)\in\R^k$ and
  $\vnu=(\nu_1,\dots,\nu_k)\in\R^k$.
\end{lemma}

\begin{proof}
  In view of the expression
  $\la{c}+\la{h}=\la{z}_{\la{c}}(\la{h})+\la{h}$, for any
  $\vt=(t_1,\dots,t_k)\in\R^k$, 
  \begin{align}
    \label{eq:35}
    \cA(\vt)&=z(\vt')+(f(\vt')+t_k)\cA_k\\
    \label{eq:35b}
    \cA(\vt')&=z(\vt')+f(\vt')\cA_k,
  \end{align}
  where $z(\vt')\in \la{z}_{\la{c}}(\la{h})$ and
  \begin{align}
    \label{eq:38}
    f(\vt')=\sum_{\ell=1}^{k-1} \frac{m_\ell+1}{m_k+1}t_\ell.
  \end{align}
  Here we observe that since $k\geq 2$,
  \begin{equation}
    \label{eq:119}
  \vt'\in(\R_{>0})^{k-1}  \implies f(\vt')>0.    
  \end{equation}

  Since the action of $\la{c}$ is via commuting $\R$-diagonalizable
  elements and $W$ is an irreducible
  $\la{z}_{\la{c}}(\la{h})+\la{h}$-module, the center
  $\la{z}_{\la{c}}(\la{h})$ of $\la{c}+\la{h}$ acts on $W$ by
  scalers. Therefore for each $\vt\in(\R_{>0})^k$, there exists a
  constants $c(\vt')\in\R$ such that
  \begin{equation}
    \label{eq:40}
    z(\vt')w=c(\vt')w,\quad\forall w\in W.
  \end{equation}
  For any $\vmu\in\Delta(W)$, there exists $0\neq w\in
  \pi_\vmu^{\cT}(W)\subset W$. Now
  \begin{align}
    \label{eq:41}
    (\vmu\cdot\vt)w =\cA(\vt)w 
    &=z(\vt')w+(f(\vt')+t_k)\cA_kw\\
    &=(c(\vt')+(f(\vt')+t_k)\mu_k)w,
  \end{align}
  where $\vmu\cdot\vt=\mu_1t_1+\dots+\mu_kt_k$. Similarly,
  \begin{equation}
    \label{eq:43}
    (\vmu'\cdot\vt')w=\cA(\vt')w=(c(\vt')+f(\vt')\mu_k)w.  
  \end{equation}
  Therefore, since $w\neq 0$,
  \begin{align}
    \label{eq:42}
    \vmu\cdot\vt&=c(\vt')+(f(\vt')+t_k)\mu_k \\
    \label{eq:42b}
    \vmu'\cdot\vt'&=c(\vt')+f(\vt')\mu_k.
  \end{align}
  Therefore, since $f(\vt')>0$, for any $\vmu,\vnu\in\Delta(W)$, 
  \begin{equation}
    \label{eq:44}
    \vmu\cdot\vt\geq \vnu\cdot\vt\Longleftrightarrow \mu_k\geq
    \nu_k\Longleftrightarrow 
    \vmu'\cdot\vt'\geq  \vnu'\cdot\vt'. 
  \end{equation}
\end{proof}

\subsubsection*{Notation} We express
$\R^{n-1}=\R^{m_k}\oplus\R^{n-1-m_k}$ and let $q:\R^{n-1}\to \R^{m_k}$
and $q_\perp:\R^{n-1}\to\R^{n-1-m_k}$ denote the associated
projections.  Let
\begin{align}
  \label{eq:45}
\begin{split}
  \frakW&=\{\omega\in \nor(u(\R^{n-1})): 
  \exp(t\cA_k)\omega\exp(-t\cA_k)\stackrel{t\to\infty}{\to} e\}\\
  &=\left\{\omega(w):= \left[
      \begin{smallmatrix} 
        1 &  0      & 0\\
        & I_{m_k} & w \\
        & & I_{n-1-m_k}
      \end{smallmatrix}
    \right]: w\in M(m_k\times (n-1-m_k),\R)\right\}.
\end{split}
\end{align}

\begin{lemma}
  \label{lema:w}
  Let $\cE\subset \R^{n-1}$ be such that
  $q(\cE)$ is a basis of $\R^{m_k}$.  Then there
  exists $\omega\in\frakW$ such that
  \begin{equation}
    \label{eq:46}
    \omega u(e)\omega\inv =u(q(e)),\quad \forall e\in\cE.
  \end{equation}
\end{lemma}

\begin{proof}
  We fix the standard basis of $\R^{n-1-m_k}$, and consider the basis
  $q(\cE)$ of $\R^{m_k}$. Then there exists a unique
  \begin{equation}
    \label{eq:6}
  w\in\End(\R^{m_k},\R^{n-1-m_k})\cong M(m_k\times(n-1-m_k),\R)  
  \end{equation}
  such that $w(q(e))=q_\perp(e)$ for all $e\in\cE$. Then $\omega(w)\in
  \frakW$ satisfies \eqref{eq:46}.
\end{proof}

\begin{lemma}
  \label{lema:wV0}
  For any $\omega\in\frakW$ the following holds:
  \begin{align}
    \label{eq:48}
    \omega(V^0(\cT)+V^-(\cT))&\subset V^0(\cT)+V^-(\cT)\\
    \label{eq:48b}
    x\in V^0(\cT)+V^-(\cT)&\Leftrightarrow \pi_0^{\cT}(x)=\pi_0^{\cT}(\omega
    x)\\
    \label{eq:49}
    \omega(V^0(\cT')+V^-(\cT'))&\subset V^0(\cT')+V^-(\cT')\\
\label{eq:49c}
    x\in V^0(\cT')+V^-(\cT'),\ \pi_{\vzero'}^{\cT'}(x)\neq 0 &\Leftrightarrow
    \pi_{\vzero'}^{\cT'}(\omega x)\neq 0.
  \end{align}
\end{lemma}

\begin{proof}
  Let $\la{w}$ denote the Lie subalgebra of $\la{g}$ associated to
  $\frakW$. Then $\la{w}$ is contained in the sum of strictly negative
  eigenspaces of $\ad(\cA(\cT))$ acting on $\la{g}$. Therefore
  \eqref{eq:48} and \eqref{eq:48b} hold.

  Similarly \eqref{eq:49} and \eqref{eq:49c} hold, because $\la{w}$ is
  contained in the sum of zero eigen\-spaces and strictly negative
  eigen\-spaces of $\ad(\cA_{\ell})$ acting on $\la{g}$ for all $1\leq
  \ell\leq k-1$.
\end{proof}

One of the crucial ingredients in the proof of
Proposition~\ref{prop:rep2} is following `Basic Lemma-I'
\cite[Proposition~4.2]{Shah:SLn}.

\begin{proposition}
\label{prop:basiclemma1}
Let $m\geq 1$ and
$\cA=\diag(m,-1,\dots,-1)\in\la{SL}(m+1,\R)$. For
any $\vx\in\R^m$, let $u(\vx)=\psmat{1& \vx \\
  0 & I_{m}}\in\la{sl}(m+1,\R)$. Let $W$ be a finite dimensional
representation of $\SL(m+1,\R)$. Let $W^-$ (respectively $W^+$) be the
sum of strictly negative (respectively positive) eigenspaces of $\cA$
and $W^0$ be the null space of $\cA$. Let $\pi_0:W\to W^0$ denote the
projection parallel to $W^- \oplus W^+$.  Let $\bar\cB_1$ be an affine
basis of $\R^m$ and $w\in W$. Suppose that
\begin{equation}
  \label{eq:8}
  u(e)w \in W^0+W^-, \quad\forall e\in\bar\cB_1.
\end{equation}
Then 
\begin{equation}
  \label{eq:12}
  \pi_0(u(e)w)\neq 0, \quad\forall e\in\bar\cB_1.
\end{equation}
\qed
\end{proposition}

\begin{corollary}
  \label{cor:H}
  In Proposition~\ref{prop:basiclemma1}, suppose further that
  $\pi_0(u(e_0)w)$ is fixed by $\SL(m+1,\R)$ for some
  $e_0\in\bar\cB_1$. Then $w$ is fixed by $\SL(m+1,\R)$.
\end{corollary}

\begin{proof}
  Let $w_0=u(e_0)w-\pi_0(u(e_0)w)$. Put
  $\bar\cB_2=\{e-e_0:\bar\cB_2\}$. Then $u(e')w_0\in V^0+V^-$ for all
  $e'\in\bar\cB_2$. Since $\bar\cB_2$ is an affine basis of $\R^m$, by
  Proposition~\ref{prop:basiclemma1}, if $w_0\neq 0$ then
  $\pi_0(w_0)\neq 0$, which contradicts the choice of $w_0$. Therefore
  $w_0=0$. Hence $u(e_0)w=\pi_0(u(e_0)w)$ is fixed by
  $\SL(m+1,\R)$. In turn, $w$ is fixed by $\SL(m+1,\R)$.
\end{proof}

\begin{proof}[Proof of Proposition~\ref{prop:rep2}]
  Let $e_0\in\cB$. We want to prove \eqref{eq:95} and \eqref{eq:120b}
  for $e_0$ in place of $e$. By replacing $v$ by $u(e_0)v$ and
  replacing every element $e\in \cB$ by $e-e_0$, without loss of
  generality we may assume that $e_0=0$. Let $\cB_1\subset \cB$
  containing $0$ such that $q(\cB_1)\setminus \{0\}$ is a basis of
  $\R^{m_k}$. By Lemma~\ref{lema:w} there exists $\omega\in \frakW$
  such that
\begin{equation}
  \label{eq:4}
  \omega u(e)\omega\inv =u(q(e)) \quad \forall e\in\cB_1.
\end{equation}

We put $v_0:=\omega v$. Then by \eqref{eq:29}, 
\eqref{eq:48} and \eqref{eq:4}, we have
\begin{equation}
  \label{eq:50}
  u(q(e))v_0=\omega(u(e)v)\subset V^0(\cT)+V^-(\cT),\quad \forall
  e\in\cB_1.
\end{equation}
By \eqref{eq:49c}, if $\pi_{\vzero'}^{\cT'}(v)\neq 0$ then
$\pi_{\vzero'}^{\cT'}(v_0)\neq 0$. By \eqref{eq:48b}, if
$\pi_{\vzero}^\cT(v_0)\neq 0$ then $\pi_{\vzero}^\cT(v)\neq 0$.
Therefore by \eqref{eq:49}, in order to prove \eqref{eq:95} and
\eqref{eq:120b}, it is enough to show that
\begin{gather}
  \label{eq:52}
  v_0=\omega v\in V^0(\cT')+V^-(\cT'),\quad\text{and}\\
    \label{eq:52c}
\text{if $\pi^{\cT'}_{\vzero}(v_0)\neq 0$ then $\pi_\vzero^{\cT}(v_0)\neq 0$.}
\end{gather}

We decompose $V$ into irreducible $(\la{c}+\la{h})$-submodules as
\begin{align}
  \label{eq:53}
  V=W_1\oplus \dots \oplus W_s.
\end{align}
For each $1\leq j\leq s$, let $P_j:V\to W_j$ denote the associate
projection, which is $(\la{c}+\la{h})$-equivariant. To show the
validity of \eqref{eq:52} and \eqref{eq:52c} it is enough to prove
that for every $1\leq j\leq s$,
\begin{gather}
  \label{eq:47}
  P_j(v_0)\in V^0(\cT')+V^-(\cT'),\quad \text{and}\\
  \label{eq:47c}
  \cA_k\cdot P_j(\pi^{\cT'}_{\vzero'}(v_0))=0.
\end{gather}
Now onwards we will fix $j$ as above, and put $W=W_j$ and $P=P_j$.

Without loss of generality we may suppose that $P(v_0)\neq 0$.  Let
\begin{align}
  \label{eq:57}
  \Delta(W,\cB_1)=\{\vnu\in\Delta(W):\pi_\vnu(P(u(q(e))v_0))\neq 0
  \quad \text{for some $e\in \cB_1$}\}.
\end{align}
We now recall that by \eqref{eq:50} we have
\begin{equation}
  \label{eq:99}
  \vnu\cdot \vt\leq 0,\quad \forall\vnu\in\Delta(W,\cB_1),\ \forall \vt\in\cT.
\end{equation}

Let $\vmu$ be the maximal element of $\Delta(W,\cB_1)\neq\emptyset$
with respect to the total preorder defined on $\Delta$; that is,
\begin{align}
  \label{eq:59}
  \vnu\cdot\vt \leq \vmu\cdot\vt,\quad \forall \vt\in\cT,\ \forall
  \vnu\in\Delta(W,\cB_1).
\end{align}
Therefore by Lemma~\ref{lema:order}, for all $\vnu\in\Delta(W,\cB_1)$
we have
\begin{align}
  \label{eq:58}
  \nu_k&\leq \mu_k \quad \text{and}\\
  \label{eq:58b}
  \vnu'\cdot\vt' & \leq \vmu'\cdot\vt, \quad \forall \vt\in\cT.
\end{align}
For $\lambda\in\R$, let
\begin{align}
  \label{eq:55}
  W_\lambda=\{w\in W:\cA_k w=\lambda w\}.
\end{align}
Then for any $\vnu=(\nu_1,\dots,\nu_k)\in\Delta$,
\begin{align}
  \label{eq:54}
  \pi_{\vnu}(P(u(q(e))v_0)\in W_{\nu_k}.
\end{align}

Hence by \eqref{eq:57}, \eqref{eq:58}, and \eqref{eq:54} we conclude
that
\begin{align}
  \label{eq:56}
  u(q(e))P(v_0)=P(u(q(e))v_0)\in \sum_{\lambda\leq\mu_k} W_{\lambda}, \quad
  \forall e\in\cB_1.
\end{align}

Let $H$ be the Lie subgroup of $G$ associated to the Lie algebra
$\la{H}$. Then $H$ is naturally isomorphic to $\SL(m_k+1,\R)$. We now
apply Proposition~\ref{prop:basiclemma1} in the case of $m=m_k$ and
$\cA=\cA_k$, $\bar\cB_1=q(\cB_1)$ and $w=P(v_0)$. Note that if $\mu_k<0$,
then by \eqref{eq:56} we have
\begin{equation}
  \label{eq:14}
  u(\bar e)w\in W^-, \quad \forall \bar e\in\bar\cB_1.
\end{equation}
Therefore the condition \eqref{eq:8} of
Proposition~\ref{prop:basiclemma1} is satisfied but its
conclusion \eqref{eq:12} fails to hold. Thus we conclude that
\begin{equation}
  \label{eq:15}
  \mu_k\geq 0.  
  \end{equation}

Therefore for any $\vt\in\cT$, by \eqref{eq:99}, we have
\begin{align}
  \label{eq:60}
  \vmu'\cdot\vt'=\vmu\cdot\vt - \mu_kt_k\leq 0.
\end{align}
Then from \eqref{eq:58b} we conclude that
\begin{align}
  \label{eq:62}
  \vnu'\cdot\vt'\leq \vmu'\cdot\vt'\leq 0, \quad\forall
  \vnu\in\Delta(W,\cB_1),\ \forall\vt'\in\cT'.
\end{align}
Now in view of \eqref{eq:57} this implies \eqref{eq:47}.

Next in order to prove \eqref{eq:47c}, suppose that $\vnu\in
\Delta(W,\cB_1)$ and $\vnu'=\vzero'$. We need to show that $\vnu=0$.

Let $i\in\N$. By \eqref{eq:62}, 
\begin{equation}
  \label{eq:64}
  0=\vnu'\cdot\vt'_i \leq \vmu'\cdot\vt'_i\leq 0.
\end{equation}
Therefore, by \eqref{eq:15} and since $\vt_{i,k}>0$,
\begin{equation}
  \label{eq:102}
  0\geq \vmu\cdot \vt_i=\vmu'\cdot\vt'_i+\mu_k t_{i,k}\geq 0,
\end{equation}
and hence
\begin{equation}
\label{eq:muk=0}
\mu_k=0.
\end{equation}
Therefore by \eqref{eq:42b} and \eqref{eq:64}, applied first to $\vnu$,
and to $\vmu$, we get
\begin{equation}
  \label{eq:5}
 f(\vt'_i)\nu_k=\vnu'\cdot\vt'_i-
 c(\vt_i')=-c(\vt_i')=f(\vt')\mu_k-\vmu'\cdot\vt'_i=0.
\end{equation}
Since $f(\vt'_i)>0$, we conclude that $\vnu=0$. 
\end{proof}

\subsection{Consequences of the Basic Lemma-II} 

For any $1\leq\ell\leq k$ and $\vt=(t_1,\dots,t_k)\in\R^k$ define
$\vt(\ell)=(t_1,\dots,t_\ell)$, and
$\cT(\ell)=(\vt_i(\ell))_{i=1}^\infty$, where
$\cT=(\vt_i)_{i=1}^\infty$ such that $\vt_i\in(\R_+)^k$, each
coordinate of $\vt_i$ tends to infinity as $i\to\infty$. By passing to
a subsequence we will further assume that
\begin{align}
  \label{eq:141}
V=V^-(\cT(\ell))\oplus
V^0(\cT(\ell))\oplus V^-(\cT(\ell))\quad (1\leq \ell\leq k).  
\end{align}

By applying Proposition~\ref{prop:rep2} repeatedly, we can decrease
$k$, and obtain the following: 

\begin{proposition}
  \label{prop:main2}
  Let $\cB$ be an affine basis of $\R^{n-1}$ and $v\in V$ be such that
\begin{equation}
    \label{eq:29b}
    u(e)v\in V^0(\cT)+V^-(\cT),\quad \forall\, e\in\cB.
  \end{equation}
  Then for any $1\leq\ell\leq k$, and any $e\in\cB$,
  \begin{gather}
    \label{eq:67}
    u(e)v\in V^0(\cT(\ell))+V^-(\cT(\ell)),\quad\text{and}\\
    \label{eq:67b}
    \text{if $\pi_{\vzero(\ell)}^{\cT(\ell)}(u(e)v)\neq 0$ then
      $\pi_\vzero^{\cT}(u(e)v)\neq 0$.}
  \end{gather}
\qed
\end{proposition}

By specializing this result to the case of $\ell=1$, we deduce the
following generalization of Proposition~\ref{prop:basiclemma1}.

\begin{corollary}
  \label{cor:rep2-main}
  Let $\cB$ be an affine basis of $\R^{n-1}$ and $v\in V$ be such that
  \begin{equation}
    \label{eq:29c}
    u(e)v\in V^0(\cT)+V^-(\cT),\quad \forall\, e\in\cB.
  \end{equation}
  Then  
  \begin{equation}
    \label{eq:37b}
    \pi_\vzero^{\cT}(u(e)v)\neq 0, \quad \forall e\in\cB.
  \end{equation}
\end{corollary}

\begin{proof}
  By \eqref{eq:67} Proposition~\ref{prop:main2} applied to $\ell=1$ we
  get 
  \begin{equation}
    \label{eq:37}
    u(e)v\in V^0(\cT(1))+V^-(\cT(1)), \quad \forall e\in\cB. 
  \end{equation}
  Take any $e_0\in\cB$ and choose $\cB_1\subset\cB$ containing $e_0$
  such that $q(\cB_1)$ is an affine basis of $\R^{m_1}$. Let $\frakW$
  be defined as in \eqref{eq:45} associated to $\cA_1$ in place of
  $\cA_k$. Then by Lemma~\ref{lema:w} there exists $\omega\in\frakW$
  such that $\omega u(e-e_0)\omega\inv=u(q(e-e_0))$ for all
  $e\in\cB_1$. Therefore
  \begin{equation}
    \label{eq:16}
    u(q(e-e_0))(\omega u(e_0)v)\in V^0(\cT(1))+V^-(\cT(1)), \quad \forall e\in\cB_1. 
  \end{equation}
  Therefore by Proposition~\ref{prop:basiclemma1} applied to $m=m_1$,
  $\cA=\cA_1$, $\bar\cB_1=\{q(e-e_0):e\in\cB_1\}$ and $w=\omega
  u(e_0)v$, we get
\begin{equation}
  \label{eq:18}
  \pi_{\vzero(1)}^{\cT(1)}(\omega u(e_0)v)\neq 0.
\end{equation}
Therefore in view of \eqref{eq:48} of Lemma~\ref{lema:wV0} for
$\frakW$ and $\cT$ defined for the case of $k=1$, we get
\begin{equation}
  \label{eq:19}
  \pi_{\vzero(1)}^{\cT(1)}(u(e_0)v)=\pi_{\vzero(1)}^{\cT(1)}(\omega
  u(e_0)v)\neq 0.
\end{equation}
Now from \eqref{eq:67b} of Proposition~\ref{prop:main2} we conclude that
$\pi_{\vzero}^{\cT}(u(e_0)v)\neq 0$. 
\end{proof}

\begin{corollary}
  \label{cor:curve-rep2-main}
  Let $\cT\subset(\R_+)^k$ be as in \S\ref{subsubsec:cTk}. Assume that
  $m_1=n-1$. Let $\phi:I=[a,b]\to \R^{n-1}$ be a differentiable curve
  whose image is not contained in a proper affine subspace of
  $\R^{n-1}$. Let $v\in V$ be such that
  \begin{equation}
    \label{eq:106}
    u(\phi(s))v\in V^0(\cT)+V^-(\cT). 
  \end{equation}
  Then $v$ is $G$-fixed.
\end{corollary}

\begin{proof}
  Given any $s_0\in I$, there exists a set
  $\{s_1,\dots,s_{n-1}\}\subset I$ such that $\cB:=\{\phi(s_i):0\leq
  i\leq n-1\}$ is an affine basis of $\R^{n-1}$. Therefore by
  \eqref{eq:67} of 
  Proposition~\ref{prop:main2} applied to $\ell=1$ we get
  \begin{equation}
    \label{eq:108}
    u(\phi(s_0))v\in V^0(\cT(1))+V^-(\cT(1)).
  \end{equation}
  Therefore by \cite[Cor.~4.6]{Shah:SLn} $G$ fixes $v$.
\end{proof}

We will now generalize the above result for all $1\leq m_1\leq n-1$.

\begin{proposition}
  \label{prop:m1}
  Let $\cT\subset(\R_+)^k$ be as \S\ref{subsubsec:cTk}. Let
  $\phi:I=[a,b]\to \R^{n-1}$ be a differentiable curve which is not
  contained in a proper affine subspace of $\R^{n-1}$. Let $v\in V$ be
  such that
  \begin{equation}
    \label{eq:140}
    u(\phi(s))v\in V^0(\cT)+V^-(\cT), \quad\forall s\in I. 
  \end{equation} 
  Then $v$ is fixed by the subgroup $Q_{m_1+1}$.
\end{proposition}

\begin{proof}
  Without loss of generality we may assume that $v\neq 0$.  In view of
  Proposition~\ref{prop:main2}
  \begin{equation}
    \label{eq:129}
    u(\phi(s))v\in V^0(\cT(1))+V^-(\cT(1)),\quad\forall s\in I.
  \end{equation}
  At this stage we will take $k=1$ and replace $\cT$ by $\cT(1)$.
  Let $L=\R^{m_1}$, $L^\perp=\R^{n-1-m_1}$, and $q:\R^{n-1}\to L$ and
  $q_\perp:\R^{n-1}\to L^\perp$ be the projections associated to the
  decomposition $\R^{n-1}=L\oplus L^\perp$.

  Take any $s_0\in I$ and put $e_0=\phi(s_0)$. Due to the hypothesis
  on $\phi$ there exists a finite set $\cE\subset \phi(I)$ containing
  $e_0$ such that $\{q(e)-q(e_0):e\in\cE\}$ is not contained in a
  union of $(m_1+1)$ proper subspaces of $L$.

  Let $\cB_1\subset\cE$ containing $e_0$ be such that the set
  $\{q(e)-q(e_0):e\in\cB_1\setminus\{e_0\}\}$ is a basis of $L$. Let
  $\frakW$ be defined as in \eqref{eq:45} for $k=1$. Then there exists
  $\omega\in \frakW$ such that $\omega u(e-e_0)\omega\inv =
  u(q(e-e_0))$ for all $e\in\cB_1$. We put $v_0=u(e_0)v$. Now by
  \eqref{eq:129} and \eqref{eq:48} for the case of $k=1$,
 \begin{align}
   \label{eq:121}
\begin{split}
   u(q(e-e_0))\omega v_0&=\omega u(e-e_0)v_0\\
                              &=\omega u(e)v 
                              \in V^0(\cA_1)\oplus V^-(\cA_1).
\end{split}
 \end{align}
 Let $H\cong\SL(m_1+1,\R)$ be the Lie group associated to the Lie
 algebra $\la{h}$ as defined through \eqref{eq:17} for $k=1$. Let
 $\cC=\{q(e-e_0):e\in\cB_1\setminus\{e_0\}\}$. Let $D_\cC$ consist of
 those $g\in Z_H(\exp(\R\cA_1))$ such that for each $e'\in\cC$, we
 have $gu(e')g\inv=u(\lambda e')$ for some $\lambda>0$;
 cf.~\cite[eq.(4.48)]{Shah:SLn}.  Then by \cite[Prop.2.3]{Shah:SLn},
 for all $e\in\cB_1$,
 \begin{align}
   \pi_0^{\cA_1}(\omega u(e)v)\neq 0 \quad\text{and}\quad 
   D_\cC\subset \Stab_G(\pi_0^{\cA_1}(\omega u(e)v)).
 \end{align}
 By \eqref{eq:48b}, applied to the case of $k=1$,
 $\pi_0^{\cA_1}(\omega v_0)=\pi_0^{\cA_1}(v_0)$, and hence
\begin{equation}
  \label{eq:63}
  \pi_0^{\cA_1}(v_0)\neq 0 \quad\text{and}\quad 
  D_\cC\subset \Stab_G(\pi_0^{\cA_1}(v_0)). 
\end{equation}
This equation holds for all choices of $\cB_1\subset\cE$ containing
$e_0$ such that $q(\cB_1)$ spans $L$; here it is important that
\eqref{eq:63} does not involve $\omega$. Therefore in view of the
hypothesis on $\cE$, by \cite[Cor.2.4]{Shah:SLn} we can deduce that
\begin{gather}
  \label{eq:125}
  \pi_0^{\cA_1}(v_0)\neq 0\quad\text{and}\\
  Z_H(\exp(\R\cA_1))\subset \Stab_G(\pi_0^{\cA_1}(v_0)).
\end{gather}

Next we want to show that
\begin{equation}
  \label{eq:126}
  u(\lambda q(\dot\phi(s)))\in \Stab_G(\pi^{\cA_1}_0(u(\phi(s)v))),
  \quad
  \forall s\in I \text{ and } \lambda\in\R.
\end{equation}

To see this, put $a(t)=\exp(t\cA_1)$ for all $t\in\R$. For
$\xi\in\R^{n-1}$, we define $a(t)\cdot\xi$ by the relation
$u(a(t)\cdot\xi)=a(t)u(\xi)a(t)\inv$. Then $a(t)\cdot
q(\xi)=e^{m_1t}q(\xi)$, $a(t)\cdot
q_\perp(\xi)=e^{(m_1-1)t}q_\perp(\xi)$, and hence
\begin{equation}
  \label{eq:131}
e^{-m_1t}a(t)\cdot\xi\stackrel{t\to\infty}{\longrightarrow} q(\xi).  
\end{equation}

Take $\lambda\in\R$ and $t_i\toinfty\infty$. Put $s_i=s+\lambda
e^{-m_1t_i}$. Then
\begin{equation}
  \label{eq:22}
\begin{split}
a(t_i)u(\phi(s_i))v
&=\pi_0^{\cA_1}(u(\phi(s_i))+a(t_i)\pi_-^{\cA_1}(u(\phi(s_i)))\\
&\toinfty \pi_0^{\cA_1}(u(\phi(s)).
\end{split}
\end{equation}
Also $a(t_i)u(\phi(s))v\toinfty \pi_0^{\cA_1}(u(\phi(s))v)$.  Since
\begin{equation}
\begin{split}
\phi(s_i)-\phi(s)
&=(s_i-s)\dot\phi(s)+O((s_i-s)^2)\\
&=e^{-m_1t_i}\dot\phi(s)+O(e^{-2m_1t_i}),
\end{split}
\end{equation}
by \eqref{eq:131},
\begin{equation}
  \label{eq:130}
  a_i\cdot(\phi(s_i)-\phi(s_0))\toinfty \lambda q(\dot\phi(s)).
\end{equation}
Therefore
\begin{equation}
\begin{split}
  \label{eq:132}
  a_iu(\phi(s_i))v&=a_iu(\phi(s_i)-\phi(s))u(\phi(s))v\\
  &=u(a_i\cdot(\phi(s_i)-\phi(s)))a_iu(\phi(s))v \\
&\toinfty u(\lambda q(\dot\phi(s)))\pi_0^{\cA_1}(u(\phi(s))v).
\end{split}
\end{equation}
Thus \eqref{eq:126} follows from \eqref{eq:22} and \eqref{eq:132}.

Due to our hypothesis on $\phi(s)$, we could choose $s_0\in I$ such
that $q(\dot\phi(s_0))\neq 0$. Let $Q$ denote the subgroup of $H$
generated by $Z_H(\exp(\R\cA_1))$ and $u(\R q(\dot\phi(s_0)))$. It may be
verified that $Q$ is a parabolic subgroup of $H$.  By \eqref{eq:125}
and \eqref{eq:126}, $Q\subset \Stab_G(\pi^{\cA_1}_0)$. Therefore we
conclude that
\begin{equation}
  \label{eq:127}
  H\subset \Stab_G(\pi^{\cA_1}_0(v_0)).
\end{equation}
Therefore $H$ fixes $\pi^{\cA_1}_0(\omega
v_0)=\pi^{\cA_1}_0(v_0)$. Now by \eqref{eq:121} and
Corollary~\ref{cor:H} we conclude that
\begin{equation}
  \label{eq:9}
  H\subset \Stab_G(\omega v_0).
\end{equation}

Let $\{e_1,\dots,e_{m_1}\}$ denote the standard basis of
$\R^{m_1}=L\subset \R^{n-1}$. Note that any $z\in Z_H(\R\cA_1)$ acts
on $e\in L$ via the relation $u(z\cdot e)=zu(e)z\inv$. This action of
$Z_H(\R\cA_1)$ surjects onto $\GL(m_1,\R)$. Therefore there exists
$z\in\Z_H(\R\cA_1)$ such that
\begin{equation}
  \label{eq:11}
\{z\cdot q(e-e_0):e\in\cB_1\}=\{0,e_1,\dots,e_{m_1}\}.   
\end{equation}

Since $\omega\in Q_{m_1+1}$, and  
\begin{equation}
  \label{eq:133}
z\omega(V^0(\cA_1)+V^-(\cA_1))=V^0(\cA_1)+V^-(\cA_1),  
\end{equation}
replacing $\phi(s)$ by the curve $\phi_1(s)$ such that 
\begin{equation}
  \label{eq:134}
  z\omega u(\phi(s)-e_0) \omega\inv z\inv  = u(\phi_1(s)),\quad\forall s\in I,
\end{equation}
and replacing $v$ by $z\omega v_0$, and $\cB_1$ by
$\{0,e_1,\dots,e_{m_1}\}$, without loss of generality we may assume
the following:
\begin{gather}
  \label{eq:135}
  H\cdot v=v, \quad\text{and}\quad 
  u(\phi(s))v\in V^0(\cA_1)+V^-(\cA_1), \quad \forall s\in I.
\end{gather}

Let $q_1:\R^{n-1}\to \R e_{m_1+1}$ denote the coordinate
projection. Put
\begin{align}
  \label{eq:142}
\cF=\{\phi(s):s\in I,\ q_1(\phi(s))\neq 0\}, \text{ then }   
L^\perp=\Span(\{q_\perp(e):e\in\cF\}).
\end{align}
Now take any $e\in \cF$. Then
there exists ${z_e}\in Z_H(\exp(\R\cA_1))$ such that
\begin{align}
  \label{eq:137}
  {z_e}\cdot q(e) = e_1, \quad\text{and}\quad {z_e}\cdot e_i=e_i\quad
  (2\leq i\leq m_k).
\end{align}
Therefore, since elements of $H$ and $u(L^\perp)$ commute, we have
\begin{equation}
  \label{eq:137b}
  {z_e}u(e){z_e}\inv = u(e_1+q_\perp(e)).
\end{equation}

Write $q_\perp(e)=(x_{m_1+1}(e),\dots,x_{n-1}(e))\in L^\perp\cong
\R^{n-1-m_1}$, and let
\begin{equation}
  \label{eq:136}
  \omega_e:= I_n + \sum_{j=m_1+1}^{n-1}
  x_j(e)E_{2,(1+j)}\in\frakW. 
\end{equation}
Then ${\omega_e}u(e_1+q_\perp(e)){\omega_e}\inv=u(e_1)$ and
${\omega_e}$ commutes with $u(e_i)$ for $i\geq 2$. Therefore by
\eqref{eq:135}, \eqref{eq:137} and \eqref{eq:137b},
\begin{align}
  \label{eq:138}
\begin{split}
V^0(\cA_1)+V^-(\cA_1)&\ni 
{\omega_e}{z_e}u(e)v=u(e_1)({\omega_e}v)\\
V^0(\cA_1)+V^-(\cA_1)&\ni
{\omega_e}{z_e}u(e_i)v=u(e_i)({\omega_e}v)\quad (2\leq i\leq m_k).
\end{split}
\end{align}
Therefore, since $\pi_0^{\cA}({\omega_e}v)=\pi_0^{\cA}(v)=v$ is fixed by
$H$, by Corollary~\ref{cor:H}
\begin{equation}
  \label{eq:139}
  H\subset\Stab_G({\omega_e}v).
\end{equation}
Now $H\cup ({\omega_e}H{\omega_e}\inv)\subset\Stab(v)$,
$\exp(\R\cA_1)\subset H$ and $\exp(t\cA_1){\omega_e}\exp(-tA_1)$
converges to the identity element as $t\to\infty$. Therefore
${\omega_e}\in\Stab(v)$ for all $e\in \cF$.

Let $Q$ be the subgroup generated by $H$ and
$\{\omega_{q_\perp(e)}={\omega_e}:e\in\cF\}$. Then $Q\subset
\Stab(v)$. By \eqref{eq:142} $Q$ is generated by $H$ and
$\omega_{L^\perp}:=\{\omega_x:x\in L^\perp\}$. Now in view of
\eqref{eq:136} it is easily verified that subgroup generated by
$Z_H(\exp(\cA_1))$ and $\omega_{L^\perp}$ contains $\frakW$. The group
generated by $\frakW$ and $u(L)$ contains $u(L^\perp)$. Therefore
\begin{equation}
  \label{eq:75}
  \Stab(v)\supset Q\supset H\cdot\frakW u(L^\perp)=Q_{m_1+1}.
\end{equation}
\end{proof}

We will need the following property of $Q_{m_1+1}$.

\begin{lemma} \label{lema:Q} Let $\vx\in\R^{m_1}\setminus \{0\}$. Then
  there is no closed proper normal subgroup of $Q_{m_1+1}$ containing
  $u(\vx)$.
\end{lemma} 

\begin{proof}
  Let $N$ be a closed normal subgroup of $Q_{m_1+1}$ containing
  $u(\vx)$. Then $u(\vx)$ belongs to $N\cap H$, which is a normal
  subgroup of $H$. Since $H\cong\SL(m_1+1,\R)$ is a simple Lie group
  with finite center, it does not contain an infinite proper normal
  subgroup. Therefore $H\subset N$. We note that if $\omega\in
  \frakW\cup u(L^\perp)$, then the closure of the group generated by
  $\omega\exp(\R\cA_1)\omega\inv$ and $\exp(\R\cA_1)$ contains
  $\omega$. Since $N$ is normal in $Q_{m_1+1}$ and
  $\exp(\R\cA_1)\subset H$, we conclude that $\frakW\cup
  u(L^\perp)\subset N$. Therefore $N=Q_{m_1+1}$.
\end{proof}

\section{Invariance under a unipotent flow}

Our aim is to prove that the measure $\mu$ as in
Corollary~\ref{cor:mu-return} is an algebraic measure. For this
purpose, we will first `stably' modify the measures $\mu_i$, and then
show that a stable modification of $\mu$ is invariant under a
unipotent flow. This will allow us to use Ratner's theorem in our
investigation. 

\subsection{Stably twisted trajectory}
Let $q:\R^{n-1}\to \R^{m_k}$ denote the projection on the span of
first $m_k$-coordinates.  We suppose that $\phi:I=[a,b]\to \R^{n-1}$
satisfies the following condition for all $s\in I$,
\begin{equation}
  \label{eq:20}
q(\dot\phi(s))\neq 0.
\end{equation}
It may be noted that since $\phi$ is an analytic curve whose image is
not contained in a proper affine subspace of $\R^{n-1}$, $\phi$
satisfies \eqref{eq:20} at all but finitely many $s\in I$.
 
Fix $w_0\in\R^{m_k}\setminus \{0\}$, and define
\begin{equation}
  \label{eq:W}
  W=\{u(sw_0):s\in\R\}.
\end{equation}
Let $Z$ denote the centralizer of $\exp(\R\cA_k)$ in
$\SL(m_k+1,\R)$. Then $Z$ acts on $\R^{m_k}$ via the correspondence
$u(z\cdot v)=zu(v)z\inv$ for all $z\in Z$ and $v\in \R^{m_k}$.  This
action is transitive on $\R^{m_k}\setminus\{0\}$. By \eqref{eq:20}
there exists an analytic function $z:I\to Z$ such that
\begin{equation}
  \label{eq:zs}
  z(s)\cdot q(\dot\phi(s)))=w_0, \quad \forall s\in I,
\end{equation}
where $\dot\phi(s)=d\phi(s)/ds$.  In view of \S\ref{subsec:cT} and
\eqref{eq:a_tau}, we set
\begin{equation}
  \label{eq:25}
a_i:=a_{\bar\vtau_i}=\exp(\cA(\vt_i)),\quad \forall i\in\N.   
\end{equation}

Like \eqref{eq:7}, for any $i\in\N$, let $\lambda_i$ be the
probability measure on $\lml$ defined by
\begin{equation}
  \label{eq:lambda_i}
  \int_{\lml} f\,d\lambda_i:=\frac{1}{\abs{I}}\int_{s\in I}
  f(z(s)a_iu(\phi(s))x_i)\,ds,\quad\forall f\in\Cc(\lml).
\end{equation}
Since $\{a_{\vtau_i}a_i\inv:i\in\N\}$ and $z(I)$ are contained in
compact subsets of $Z$, from Theorem~\ref{thm:nondiv} we deduce that
there exists a probability measure $\lambda$ on $\lml$ such that,
after passing to a subsequence, $\lambda_i\to\lambda$ in the space of
probability measures on $\lml$ with respect to the weak-$\ast$
topology.

\begin{theorem} 
  \label{thm:W-invariant}
  The measure $\lambda$ is $W$-invariant.
\end{theorem}

\begin{proof}
  Let $q_\perp:\R^{n-1}\to \R^{n-1-m_k}$ denote the projection on the
  last $(n-1-m_k)$-coordinates. Let $\alpha_i=\exp(\sum_{j=1}^{n-1} \tau'_{i,j} +
  \tau'_{i,m_k})$. Then for any $\xi\in \R^{n-1}$,
  \begin{equation}
    \label{eq:111}
    v=q(\xi)+q_\perp(\xi),\quad a_i\cdot q(\xi)=\alpha_iq(\xi)\quad\text{and}\quad \alpha_i\inv(a_i\cdot
    q_\perp(\xi))\toinfty 0.    
  \end{equation}

  Let $t\in\R$. Take any $\epsilon>0$. For $i\in\N$, let
  $N_i:=[\epsilon\abs{I}\alpha_i]\in\N$.  Then 
\begin{equation}
\label{eq:Ni}
\alpha_i/N_i\toinfty (\epsilon\abs{I})\inv \quad\text{and}\quad
\alpha_i/N_i^2\toinfty 0.
\end{equation} 
We partition $I=\cup_{r=1}^{N_i} I_r$, where $I_r=[s_r,s_{r+1}]$ and
$s_{r+1}-s_r=\abs{I}/N_i$.  Let
  \begin{equation}
    \label{eq:10}
\begin{split}
\psi_r(s)&:=\phi(s_r)+(s-s_0)\dot\phi(s_r), \quad \forall
s\in\R,\text{ then}\\
    \phi(s)&=\psi_r(s)+\varepsilon_r(s) \quad\text{and}
    \quad \varepsilon_r(s)=O(N_i^{-2}),\quad\forall s\in I_r. 
  \end{split}
\end{equation}
By \eqref{eq:111} and \eqref{eq:Ni}, $\sup_{s\in I_r}\norm{a_i\cdot
  \varepsilon_r(s)}\toinfty 0$. Since $z(\cdot)$ is continuous and
bounded, for all large $i$ and $1\leq r\leq N_i$,
  \begin{equation}
    \label{eq:118}
    \abs{f(z(s)a_iu(\phi(s))x_i)-f(z_ra_iu(\psi_r(s))x_i)}\leq
    \epsilon, \quad \forall s\in I_r,
  \end{equation}
  where $z_r=z(u(s_r))$, and the same holds for $f^{u({t} w_0)}$ in
  place of $f$, where $f^{u({t} w_0)}(x):=f(u({t} w_0)x)$ for all
  $x\in L/\Lambda$. Therefore
\begin{equation}
  \label{eq:122}
  \Abs{\int_{\lml}f(x)\,d\lambda_i(x) - \frac{1}{\abs{I}}\sum_{i=1}^{N_i} \int_{I_r}
    f(z_ra_iu(\psi_r(s))x_i)\,ds}\leq \epsilon,
\end{equation}
and the same for $f^{u({t} w_0)}$ in place of $f$.

Next, for any $s\in I_r$, by \eqref{eq:zs} and \eqref{eq:111},
\begin{align}
  \label{eq:116}
\begin{split}
u({t} w_0)z_ra_iu(\psi_r(s))&=z_ru({t}
q(\dot\phi(s_r)))a_iu(\psi_r(s))\\
&=z_ra_iu({t} \alpha_i\inv q(\dot\phi(s_r)))u(\psi_r(s))\\
&=z_ra_iu(-{t}\alpha_i\inv
q_\perp(\dot\phi(s_r)))u(\psi_r(s)+t\alpha\inv \dot\phi(s_r))\\
&=u(z_ra_i\cdot (-t\alpha_i\inv q_\perp(\dot\phi(s_r))))
z_ra_iu(\psi_r(s+{t}\alpha_i\inv)).
\end{split}
\end{align}
By \eqref{eq:111}, $\sup_{s\in I_r} \norm{z_ra_i\cdot (-t\alpha_i\inv
  q_\perp(\dot\phi(s_r)))}\toinfty 0$. Hence for large enough $i$,
\begin{equation}
  \label{eq:120}
  \sum_{r=1}^{N_i} \int_{I_r} \Abs{f(u({t} w_0)z_ra_i\psi_r(s)x_i) -
    f(z_ra_i\psi_r(s+t\alpha_i\inv)x_i)}\,ds\leq \epsilon \abs{I}. 
\end{equation}
Now by \eqref{eq:Ni}
\begin{equation}
\label{eq:123}
\begin{split}
  &\sum_{r=1}^{N_i} \Abs{\int_{I_r} f(z_ra_iu(\psi_r(s))x_i)\,ds -
    \int_{I_r} f(z_ra_iu(\psi_r(s+t\alpha_i\inv)x_i))\,ds}\\
  &\leq N_i(2\norm{f}_\infty{t}\alpha_i\inv)\toinfty
  2{t}\norm{f}_{\infty}\epsilon\abs{I}.
\end{split}
\end{equation}
For all large $i$, combining \eqref{eq:122}, \eqref{eq:120} and
\eqref{eq:123}:
\begin{equation}
  \label{eq:124}
\Abs{\int_{\lml} f(u({t} w_0)y)\,d\lambda_i(y)-\int_{\lml}
  f(y)\,d\lambda_i(y)}\leq (3+2t\norm{f}_\infty)\epsilon.
\end{equation}
Since $\epsilon>0$ is arbitrary, $\lambda$ is $u({t} w_0)$-invariant.
\end{proof}

\section{Ratner's theorem and dynamical behaviour of translated
  trajectories near singular sets} \label{sec:Ratner}

Next we will analyze the measure $\lambda$ using Ratner's description
of ergodic and invariant measures for unipotent flows.

For $\sH$ be as defined in \S\ref{subsec:sH} and $W$ be an
$\Ad$-unipotent one-parameter subgroup of $G$. For $H\in\sH$, define
\begin{align}
  N(H,W)=\{g\in G: g\inv Wg\subset H\} \quad\text{and}\quad
  S(H,W)=\tcup{\substack{F\in\sH\\F\subsetneq H}}N(F,W).
\end{align}

Let $\pi:L\to \lml$ denote the natural quotient map. By Ratner's
theorem~\cite{R:measure}, as explained in \cite[Theorem
2.2]{Moz+Shah:limit}:
\begin{theorem}[Ratner]
  \label{thm:Ratner}
  Given a $W$-invariant probability measure $\lambda$ on $\lml$, there
  exists $H\in\sH$ such that
  \begin{equation}
    \label{eq:lambda-H} 
    \lambda(\pi(N(H,W))>0 \quad\textrm{and} \quad \lambda(\pi(S(H,W))=0.
  \end{equation}
  And almost all $W$-ergodic components of the restriction of
  $\lambda$ to $\pi(N(H,W))$ are of the form $g\mu_H$, where $g\in
  N(H,W)\setminus S(H,W)$ and $\mu_H$ is a finite $H$-invariant
  measure on $\pi(H)\cong H/H\cap\Lambda$.

  In particular if $H$ is a normal subgroup of $L$ then $\lambda$ is
  $H$-invariant.\qed
\end{theorem}

\subsection{Algebraic criterion for zero limit measure on singular
  sets}

Similar to the nondivergence criterion given by
Proposition~\ref{prop:return}, the next result provides a criterion
for `non-accumulation on singular sets' in terms of linear actions of
groups; it is also referred by `linearization technique'. Let the
notation be as in \S\ref{subsec:sH}. Let
$w_0\in\Lie(W)\setminus\{0\}$. Let $\sA=\{v\in V: v\wedge w_0=0\}$. Then
\begin{equation}
  \label{eq:A}
  N(H,W)=\{g\in L: g\cdot p_H\in \sA\}.
\end{equation}

The following linearization statement from \cite[Prop.~4.4]{Shah:son1}
uses the fact that $\phi$ is analytic;
cf.~\cite{R:uniform,Dani+Mar:limit,Moz+Shah:limit}.

\begin{proposition}
  \label{prop:main3}
  Let $C$ be any compact subset of $N(H,W)\setminus S(H,W)$. Let
  $\epsilon>0$ be given. Then there exists a compact set
  $\cD\subset\sA$ such that given any neighbourhood $\Phi$ of $\cD$ in
  $V$, there exists a neighbourhood $\cO$ of $\pi(C)$ in $\lml$
  such that for any $h_1,h_2\in L$, and a subinterval $J\subset I$,
  one of the following holds:
  \begin{enumerate}
  \item[a)] There exists $\gamma\in\Lambda$ such that
    $(h_1z(s)u(\phi(s))h_2\gamma)p_H\in\Phi$, $\forall s\in J$.
  \item[b)] $\abs{\{s\in J: \pi(h_1z(s)u(\phi(s))h_2)\in\cO\}}\leq
    \epsilon\abs{J}$.
  \end{enumerate}
\qed
\end{proposition}

Just as in the Proof of Theorem~\ref{thm:nondiv} we will apply this
criterion to obtain an algebraic condition leading to the hypothesis
of Corollary~\ref{cor:rep2-main}.

\begin{theorem}
  \label{thm:lambda}
  Suppose that there is no proper closed subgroup $H$ of $L$
  containing $\rho(Q_{m_1+1})$ such that the orbit $Hx_0$ is closed
  and admits a finite $H$-invariant measure. Let
  $\{\lambda_i\}_{i=1}^\infty$ be the sequence of measures as defined
  by \eqref{eq:lambda_i}. Then $\lambda_i\toinfty\lambda$ in the space
  of probability measures on $\lml$ and $\lambda$ is $L$-invariant.
\end{theorem}

\begin{proof}
  Earlier using Theorem~\ref{thm:nondiv} we have shown that after
  passing to a subsequence, $\lambda_i\to\lambda$ in the space of
  probability measures on $\lml$, and by Theorem~\ref{thm:W-invariant}
  $\lambda$ is invariant under the $\Ad$-unipotent one-parameter
  subgroup $W$. In order to complete the proof it is enough to show
  that any such limiting measure $\lambda$ is $L$-invariant. For
  notational convenience we will identify any $g\in G$ with
  $\rho(g)\in L$.

  By Theorem~\ref{thm:Ratner} there exists $H\in\sH$ such that
  \begin{equation}
    \label{eq:NHW}
    \lambda(\pi(N(H,W))>0 \quad\textrm{and} \quad \lambda(\pi(S(H,W))=0.
  \end{equation}

  Let $C$ be a compact subset of $N(H,W)\setminus S(H,W)$ such that
  $\lambda(C)>\epsilon$ for some $\epsilon>0$. Let $g_i\toinfty g_0$
  be as sequence in $G$ such that $x_0=\pi(g_0)$ and $x_i=\pi(g_i)$
  for all $i$. Then given any neighbourhood $\cO$ of $\pi(C)$ in
  $\lml$, there exists $i_0>0$ such that for all $i\geq i_0$, we have
  $\lambda_{i}(\cO)>\epsilon$ and hence
  \begin{equation}
    \label{eq:accum}
    \frac{1}{\abs{I}}\abs{\{s\in I:z(s)a_{i}u(\phi(s))x_i
      =\pi(a_{i}z(s)u(\phi(s))g_i)\in\cO\}} > \epsilon.
  \end{equation}

  Let $\cD\subset \sA$ be as in the statement of
  Proposition~\ref{prop:main3}. Choose any compact neighbourhood
  $\Phi$ of $\cD$ in $V$. Then there exists a neighbourhood $\cO$ of
  $\pi(C)$ in $\lml$ such that one of the statements (a) or (b) of
  Proposition~\ref{prop:main3} holds for $J=I$, and any $h_1=a_i$ and
  $h_2=g_i$. For any $i>i_0$, (a) cannot hold due to \eqref{eq:accum},
  and hence (b) must hold; that is, there exists $\gamma_i\in\Lambda$
  such that
  \begin{equation}
    \label{eq:30-gammai}
    (z(s)a_{i}u(\phi(s))g_i\gamma_i)p_H=(a_{i}z(s)u(\phi(s))g_i\gamma_i)p_H\in
    \Phi, \quad\forall s\in I.
  \end{equation}
  Let $\Phi_1=\{z(s)\inv:s\in I\}\Phi$. Then $\Phi_1$ is contained in
  a compact subset of $V$, and the following holds:
\begin{equation}
  \label{eq:32-i}
  a_{i}u(\phi(s))(g_i\gamma_i)p_H\in \Phi_1, \quad \forall s\in I,\
  \forall i>i_0. 
\end{equation} 

Let $\norm{\cdot}$ be a norm on $V$. First suppose that after passing
to a subsequence,
\begin{equation}
  \label{eq:30-i}
r_i:=\norm{\gamma_ip_H}\to \infty \quad\text{as $i\to\infty$.}
\end{equation}
Then $v_i:=\gamma_ip_H/r_i\toinfty v$ for some $v\in V$, $\norm{v}=1$. Let
$R=\sup\{\norm{w}:w\in \Phi_1\}$. Then by \eqref{eq:32-i}
\begin{equation}
  \label{eq:61}
  a_iu(\phi(s))g_iv_i\leq R/r_i, \quad \forall s\in I,\ \forall i>i_0. 
\end{equation}
Since $R/r_i\toinfty 0$, and $g_iv_i\toinfty g_0v$, we conclude that
\begin{equation}
  \label{eq:69}
  u(\phi(s))g_0v_0\subset V^-(\cT), \quad\forall s\in I.
\end{equation}
Since there exists a finite set $F\subset I$ such that
$\cB=\{\phi(s):s\in F\}$ is an affine basis of $\R^{n-1}$,
\eqref{eq:69} satisfies the condition~\eqref{eq:29c} of
Corollary~\ref{cor:rep2-main} but contradicts its conclusion
\eqref{eq:37b}. Thus \eqref{eq:30-i} fails to hold after passing to a
subsequence. Therefore the set $\{\gamma_i p_H:i\in\N\}$ is
bounded. It is discrete by Proposition~\ref{prop:discrete}. Hence it
is a finite set. Therefore by passing to a subsequence there exists
$\gamma\in\Gamma$ such that
\begin{equation}
  \label{eq:70}
  \gamma_i p_H=\gamma p_H, \quad \forall i\in\N.
\end{equation}
Therefore by \eqref{eq:32-i} we get
\begin{equation}
  \label{eq:71}
    a_{i}u(\phi(s))g_i(\gamma p_H)\subset \Phi_1, \quad \forall s\in I,\
  \forall i\in\N.
\end{equation}\

Let $\pi_+^{\cT}:V\to V^+(\cT)$ be the projection parallel to
$V^0(\cT)+V^-(\cT)$. First suppose that
$\pi^{\cT}_+(u(\phi(s))g_0\gamma p_H\neq 0$ for some $s_0\in I$. Then
there exists $c>0$ and $i_1\in\N$ such that
$\norm{\pi_+^{\cT}(u(\phi(s_0))g_i\gamma p_H)}\geq c$ for all $i\geq
i_1$. But then $\norm{a_iu(\phi(s_0)g_i\gamma p_H)}\toinfty\infty$,
which contradicts \eqref{eq:71}. Therefore
\begin{equation}
  \label{eq:72}
  u(\phi(s))(g_0\gamma p_H)\subset V^0(\cT)\oplus V^-(\cT), 
  \quad \forall s\in I.
\end{equation}
Therefore by Proposition~\ref{prop:m1}, $Q_{m_1+1}$ stabilizes
$g_0\gamma p_H$. By \eqref{eq:stab},
\begin{equation}
  \label{eq:74}
g_0\inv Q_{m_1+1}g_0\subset \noroneL(H)=\Stab_L(\gamma p_H). 
\end{equation}

Since $\Lambda p_H$ is discrete, $\Lambda \noroneL(H)$ is a closed
subset of $L$. Hence $\pi(\noroneL(H))$ is closed in $\lml$.  By
\cite[Thm.~2.3]{Shah:uniform} there exists a closed subgroup $H_1$ of
$\noroneL(H)$ containing all $\Ad$-unipotent one-parameter subgroups
of $L$ contained in $\noroneL(H)$ such that $H_1\cap\Lambda$ is a
lattice in $H_1$ and $\pi(H_1)$ is closed. Since $Q_{m_1+1}$ is
generated by unipotent one-parameter subgroups of $\SL(n,\R)$, by
\eqref{eq:74}, $g_0\inv Q_{m_1+1}g_0\subset H_1$. Thus
$Q_{m_1+1}\subset g_0H_1g_0\inv$ and
$(g_0H_1g_0\inv)x_0=g_0\pi(H_1)$ is closed and admits a finite
$g_0Hg_0\inv$-invariant measure. Hence by the hypothesis of the
theorem, $g_0H_1g_0\inv=L$. Therefore $H$ is a normal subgroup of
$L$. Therefore by Theorem~\ref{thm:Ratner}, $\lambda$ is $H$-invariant.

By \eqref{eq:NHW} there exists $g\in \pi(N(H,W))\neq \emptyset$. Then
$W\subset gHg\inv=H$. Thus $W\subset Q_{m_1+1}\cap H$, which is a
normal subgroup of $Q_{m_1+1}\cap H$. Hence by Lemma~\ref{lema:Q} we
have $Q_{m_1+1}\subset H$. Since $Hx_0=\pi(Hg_0)=g_0\pi(H)$ is closed
and admits a finite $H$-invariant measure, by our hypothesis
$H=L$. Therefore $\lambda$ is $L$-invariant.
\end{proof}

\begin{proof}[Proof of Theorem~\ref{thm:main-action}]  
Let the notation be as in \S\ref{subsec:cT}. Then
  \[
\vtau_i-\bar\vtau_i\toinfty \tilde\vtau_0:=
  (\tau(1),\dots,\tau(n-1))\in\R^{n-1}, \quad \text{as $i\to\infty$}.
\]
Therefore $a_{\vtau_i}a_{\bar\vtau_i}\inv \toinfty
a_{\tilde\vtau_0}$. Hence in $G/Q_{m_1+1}$,
  \begin{equation}
    \label{eq:110}
  a_{\vtau_i}Q_{m_1+1}\toinfty a_{\tilde\vtau_0}Q_{m_1+1}=a_{\vtau_0}Q_{m_1+1}.  
  \end{equation}
  Therefore to prove \eqref{eq:9L}, it is enough to prove the theorem
  in the case of $\tau_i=\bar\tau_i$.

  We put $x_i=x_0$ for all $i$. As noted before there exists a
  smallest closed subgroup $H$ of $L$ containing $\rho(Q_{m_1+1})$
  such that the orbit $Hx_0$ is closed and admits a finite
  $H$-invariant measure. Therefore without loss of generality we may
  replace $L$ by $H$. Now to complete the proof of the theorem, we
  only need to prove that $\mu$ is $L$-invariant.

  Since $\phi$ is analytic, the condition~\eqref{eq:zs} fails to hold
  only for finitely many points, and it is straightforward to reduce
  the proof of the theorem to the case where \eqref{eq:zs} holds for
  all $s\in I$. Now the difference between $\mu_i$ and $\lambda_i$ is
  only through $\{z(s):s\in I\}$. Given $\epsilon>0$, there exists
  $\delta>0$ such that if $J=[s_1,s_2]\subset I$ and $0<s_2-s_1<\delta$
  then $\abs{f(z(s_1)\inv z(s)x)-f(x)}\leq \epsilon$ for all $s\in J$ and
  $x\in \lml$. We define $\lambda_i^J$ by putting $J$ in place of $I$
  in \eqref{eq:lambda_i}, similarly we define $\mu_i^J$. Then by
  Theorem~\ref{thm:lambda}, $\lambda_i^J\to\lambda_L$, where
  $\lambda_L$ is the unique $L$-invariant probability measure on
  $\lml$. Since $a_iz(s)=z(s)a_i$, we deduce that
  \begin{equation}
    \label{eq:112}
    \Abs{\int f\,d\mu_i^J-\int f(z(s_1)\inv x)\,d\lambda_L(x)}\leq \epsilon,
  \end{equation}
  for all large $i$.  Since $\lambda_L$ is $z(s_1)$-invariant, the
  second integral is same as $\int f\,d\lambda_L$. Now partitioning
  $I$ into finitely many $J$'s with $\abs{J}\leq \delta$, we deduce
  \begin{equation}
    \label{eq:128}
    \Abs{\int f\,d\mu_i-\int f\,d\lambda_L}\leq \epsilon,
  \end{equation}
for all large $i$. Thus $\mu_i\toinfty\lambda_L$. 
\end{proof}

\begin{proof}[Proof of Theorem~\ref{thm:main-uniform}]
  The above proof applies to this case also. Here we are given that
  $x_i\to x_0$ is a convergent sequence and there is no proper closed
  subgroup $H$ of $G$ containing $\rho(Q_{m_1+1})$ such that $Hx_0$ is
  closed and admits a finite $H$-invariant measure. Therefore there is
  no need to replace $H$ by $L$ as in the above proof.
\end{proof}

The proof of Theorem~\ref{thm:uniform:action} can be obtained by
combining the ideas of the proof of Theorem~\ref{thm:lambda},
Proposition~\ref{prop:rep2} and the general strategy of the proof of
\cite[Theorem~3]{Dani+Mar:limit}.


\begin{thebibliography}{99}
\bibitem{Baker:curves} 
R.~C. Baker.  
\newblock Dirichlet's theorem on Diophantine approximation.  
\newblock {\em Math. Proc. Cambridge Philos. Soc.}, 83(1):37--59, 1978.

\bibitem{Bugeaud:poly}
Y. Bugeaud.
\newblock Approximation by algebraic integers and Hausdorff
dimension.
\newblock {\em J. London Math. Soc. (2)}, 65(3):547--559, 2002.  

\bibitem{Cassels:Geometry of Numbers} J.~W.~S. Cassels. 
\newblock{\em An introduction to the geometry of numbers}.
\newblock Die Grundlehren der mathematischen Wissenschaften, Band 99.
Springer-Verlag, Berlin-New York,  1971. viii+344 pp.
		

\bibitem{Dani:div} S.~G. Dani.  \newblock Divergent trajectories of
  flows on homogeneous spaces and Diophantine approximation.
  \newblock {\em J. Reine Angew.\ Math.}  359:55--89, 1985.

\bibitem{Dani+Mar:asymptotic} S.~G. Dani and G.~A. Margulis.  \newblock
  Asymptotic behaviour of trajectories of unipotent flows on
  homogeneous spaces.  \newblock {\em Proc. Indian
    Acad. Sci. Math. Sci.}, 101(1):1--17, 1991.

\bibitem{Dani+Mar:limit}
S.~G. Dani and G.~A. Margulis.
\newblock Limit distributions of orbits of unipotent flows and values of
  quadratic forms.
\newblock In {\em I. M. Gelfand Seminar}, pages 91--137. Amer. Math. Soc.,
  Providence, RI, 1993.

\bibitem{Davenport+Schmidt:Dirichlet} H. Davenport and W.~M.Schmidt
\newblock Dirichlet's theorem on diophantine approximation. II.
\newblock {\em Acta Arith.}, 16:413--424, 1969/1970.

\bibitem{DS:curve} H. Davenport and W.~M. Schmidt.
\newblock Dirichlet's theorem on diophantine approximation.
\newblock {\em Symposia Mathematica}, Vol. IV (INDAM, Rome, 1968/69), 
 pp. 113--132, Academic Press, London, 1970.

\bibitem{Dodson:manifolds}
M.~M. Dodson, B.~P. Rynne, J.~A. Vickers.
\newblock Dirichlet's theorem and Diophantine approximation on manifolds.
\newblock{\em J. Number Theory}, 36(1):85--88, 1990.

\bibitem{Klein+Mar:Annals98} D.~Y. Kleinbock and G.~A. Margulis.
  \newblock Flows on homogeneous spaces and {D}iophantine
  approximation on manifolds.  \newblock {\em Ann. of Math. (2)},
  148(1):339--360, 1998.

\bibitem{Kleinbock+Weiss:Dirichlet} Dmitry Kleinbock and Barak Weiss.
  \newblock Dirichlet's theorem on Diophantine approximation and
  homogeneous flows. \newblock {\em Journal of Modern Dynamics (JMD)},
  2(1):43--62, 2008. 

\bibitem{Moz+Shah:limit}
Shahar Mozes and Nimish A. Shah.
\newblock On the space of ergodic invariant measures of unipotent flows.
\newblock {\em Ergodic Theory Dynam. Systems}, 15(1):149--159, 1995.

\bibitem{R:measure} Marina Ratner. \newblock On Raghunathan's
  measure conjecture.  \newblock {\em Ann. of Math. (2)},
  134(3):545--607, 1991.

\bibitem{R:uniform} Marina Ratner.  \newblock Raghunathan's
  topological conjecture and distributions of unipotent flows.
  \newblock {\em Duke Math. J.}, 63(1):235--280, 1991.

\bibitem{Schmidt:Springer-lect} Wolfgang M. Schmidt. \newblock {\em
    Diophantine approximation}. Lecture Notes in Mathematics,
  785. Springer, Berlin, 1980. x+299 pp.

\bibitem{Shah:uniform}
Nimish~A. Shah.
\newblock Uniformly distributed orbits of certain flows on homogeneous spaces.
\newblock {\em Math. Ann.}, 289(2):315--334, 1991.

\bibitem{Shah:horo} 
Nimish~A. Shah.
\newblock {Limit distributions of expanding translates of certain
orbits on homogeneous spaces}.
\newblock {\em Proc.\  Indian Acad.\  Sci.\  (Math.\  Sci.)}, 106:105--125,
1996.

\bibitem{Shah:son1} 
Nimish~A. Shah.  \newblock Limiting distributions
  of curves under geodesic flow on hyperbolic manifold. 25 pages. 
\newblock arXiv:0708.4093v1 

\bibitem{Shah:SLn}
Nimish~A. Shah. \newblock Equidistribution of expanding translates of
curves and Dirichlet's theorem on Diophantine approximation. 26
pages. \newblock arxiv:0802.3278v1
\end{thebibliography}
\end{document}